\documentclass[11pt,leqno]{amsart}
\usepackage{amsmath}
\usepackage{amssymb}
\usepackage{enumerate}
\usepackage{multirow,array}
\usepackage{tablefootnote}
\usepackage{amsthm}
\usepackage{graphics,epsfig,color}

\def\ba{\begin{array}}
\def\ea{\end{array}}
\def\beq{\begin{equation}}
\def\eeq{\end{equation}}
\def\bea{\begin{eqnarray}}
\def\eea{\end{eqnarray}}
\def\beann{\begin{eqnarray*}}
\def\eeann{\end{eqnarray*}}

\def\dim{\textup{dim}}
\def\inte{\textup{int}}
\def\relint{\textup{relint}}

\renewcommand{\arraystretch}{1.2}
\newtheorem{theorem}{Theorem}[section]
\newtheorem{lemma}{Lemma}[section]

\newtheorem{proposition}{Proposition}[section]
\newtheorem{definition}{Definition}[section]
\newtheorem{remark}{Remark}[section]
\newtheorem{example}{Example}[section]

\def\ba{\begin{array}}
\def\ea{\end{array}}
\def\beq{\begin{equation}}
\def\eeq{\end{equation}}
\def\bea{\begin{eqnarray}}
\def\eea{\end{eqnarray}}
\def\beann{\begin{eqnarray*}}
\def\eeann{\end{eqnarray*}}

\def\aff{\textup{aff}}
\def\lin{\textup{lin}}

\def\dim{\textup{dim}}
\def\dom{\textup{dom}}
\def\arg{\textup{arg}}

\def\arg{\textup{arg}}
\def\inte{\textup{int}}
\def\Pr{\textup{Pr}}
\def\relint{\textup{relint}}

\title[A Utility Theory Based Approach to Robustness]{\bf  A Utility Theory Based Interactive Approach to Robustness in Linear Optimization}

\author{Mehdi Karimi \and Somayeh Moazeni \and Levent Tun\c{c}el}
\date{\today}
\thanks{Mehdi Karimi: (m7karimi@uwaterloo.ca) Department of
Combinatorics and Optimization, Graduate Student, University
of Waterloo, Waterloo, Ontario N2L 3G1, Canada. Research of this
author was supported in part by a Discovery Grant from NSERC and by ONR
Research Grant N00014-12-10049.\\
Somayeh Moazeni: (smoazeni@stevens.edu)
School of Systems and Enterprises,
Stevens Institute of Technology, Babbio Center,
1 Castle Point Terrace on Hudson
Hoboken, NJ 07030, U.S.A.  Research of this
author was supported in part by Discovery Grants from NSERC.\\
Levent Tun\c{c}el: (ltuncel@uwaterloo.ca) Department of
Combinatorics and Optimization, Faculty of Mathematics, University
of Waterloo, Waterloo, Ontario N2L 3G1, Canada. Research of this
author was supported in part by Discovery Grants from NSERC and by ONR
Research Grant N00014-12-10049.}

\begin{document}

\begin{abstract}
We treat uncertain linear programming problems by utilizing the notion
of weighted analytic centers and notions from the area of
multi-criteria decision making.
After introducing our approach, we develop
interactive cutting-plane algorithms for robust optimization,
based on concave and
quasi-concave utility functions. 
In addition to practical advantages, due to the flexibility
of our approach, we are able to prove that under a theoretical framework due to Bertsimas and Sim \cite{price}, which establishes the existence of certain convex formulation of robust optimization problems, the robust optimal solutions
generated by our algorithms are at least as desirable to the decision maker
as any solution generated by many other robust optimization algorithms in the theoretical framework.
We present some probabilistic
bounds for feasibility of robust solutions
and evaluate our approach by means of
computational experiments.
\end{abstract}
\maketitle
\section{Introduction} \label{intro}
Optimization problems are widespread in real life decision making
situations. However, data perturbations as well as uncertainty in at least part of the
data are very difficult to avoid in practice.
Therefore, in most cases we have to deal with the reality that
some aspects of the data of the optimization problem at hand are
uncertain. This uncertainty is caused by many sources such as
forecasting, or approximations in the design of mathematical models,  or data approximation, or noise in measurements.
In order to handle optimization problems under
uncertainty, several techniques have been proposed. The most common, widely-known
approaches are
\begin{itemize}
\item{\textbf{Sensitivity analysis:}} typically, the influence of data uncertainty
is initially ignored, and then the obtained solution is justified/analyzed based on
the data perturbations \cite{bon}. 
\item{\textbf{Chance constrained programming:}} we use some stochastic models of
uncertain data to replace the deterministic constraints by their
probabilistic counterparts \cite{chance1,chance2,chance3,R-GO-6}. It is a natural way
of converting the uncertain optimization problem into a deterministic one.
However, most of the time the result is a computationally intractable
problem \cite{robust-book}. 
\item{\textbf{Stochastic programming:}} the goal is to find a solution that is feasible
for all (or almost all) possible instances of the data and to optimize
the expectation of some function of the decisions and the random variables \cite{sto-shapiro}.
\item{\textbf{Robust optimization:}} robust optimization is the method that is most closely related to our approach.
Generally speaking, robust optimization can be applied to any optimization
problem where the uncertain data can be separated from the problem's structure.
Having been heavily studied for convex optimization problems  \cite{robust-book, soyster,uncertain,jadid,rco, adjust, price,
bsdis,new, norm, mulvey}, robust optimization is also applicable to discrete \cite{R-GO-1, R-GO-2,R-GO-3} and more general nonconvex optimization problems \cite{ R-GO-7}. Robustness can be achieved by 
solving the robust counterpart or utilizing other unconventional methods such as simulated annealing algorithm \cite{R-GO-5}. Our focus
in this paper is on uncertain linear programming problems.
Uncertainty in the data means that the exact values of the data are
not known, at the time when the solution has to be determined. In
robust optimization framework, uncertainty in the data is described
through \emph{uncertainty sets}, which contain all
possible values that may be realized for the uncertain parameters.
Generally speaking, the distinction between robust optimization and stochastic programming 
is that robust optimization does not require the specification of the exact distribution. Stochastic 
programming performs well when the distributions of the uncertainties are exactly known, and 
robust optimization can be very useful when there is little information about those distributions. 
\end{itemize}
Since the interest in robust formulations was revived in the 1990s,
many researchers have introduced new formulations for robust optimization
framework in linear programming and general convex programming \cite{soyster,uncertain,jadid,rco, adjust,price,
bsdis,new, norm, mulvey}. Ben-Tal
and Nemirovski \cite{uncertain,jadid}  provided some of the first
formulations for  robust LP with detailed mathematical analysis.
 Bertsimas and Sim \cite{price} proposed an approach that offers
control on the degree of conservatism for every constraint as well as the objective function.
Bertsimas et al. \cite{norm} characterize the robust counterpart of an LP problem with
uncertainty set described by an arbitrary norm. By choosing appropriate norms,
they recover the formulations proposed in the above papers \cite{uncertain,jadid,norm}.

The goal of classical robust optimization is to find a solution that
is capable to cope best of all with $\emph{all}$ realizations of the
data from a given (usually bounded) uncertainty set \cite{robust-book,boydnemi}.
By the classical definition of robustness
\cite{robust-book,rco,new,ghaoui}, a \emph{robust optimal solution} is the solution of the following problem:
\begin{eqnarray}\label{classicp}
\ \ \max_{x\in{\mathbb R}^{n}} \left \{\inf_{\tilde{c}\in\mathcal C}
\langle \tilde{c}, x\rangle : \tilde{A}x\leq\tilde{b} \ , \forall
\tilde{b}\in\mathcal B, \forall
\tilde{A}\in\mathcal A \right\},
\end{eqnarray}
where $\mathcal C$, $\mathcal A$, and $\mathcal B$ are given uncertainty sets for
$\tilde{c}$, $\tilde{A}$, and $\tilde{b}$, respectively. Throughout this paper,
we refer to the formulation of \eqref{classicp} as \emph{classical robust formulation}.
\subsection{Some drawbacks of robust optimization}
Classical robust optimization is a powerful method to deal
with optimization problems with uncertain data, however, we can raise some criticisms.
One of the assumptions for robust
optimization is that the uncertainty set must be precisely specified
before solving the problem. Even if the
uncertainty is only in the right-hand-side, expecting the Decision Maker (DM) to construct accurately
an ellipsoid or even a hypercube for the uncertainty set may not always be reasonable.
Recently, a new approach has been proposed, called \emph{distributionally robust} optimization, that tries
to cover the gap between robust optimization and stochastic programming \cite{distrib, kuhn, shapiro-1}. In this approach, 
one seeks a solution that is feasible for the worst-case probability distribution in a set of possible distributions. In a recent paper, Shapiro \cite{shapiro-1} studied distributionally robust stochastic programming in a scenario that the uncertainty set of probability measures is ``close" to a reference measure. 
It is mentioned in \cite{kuhn} and also emphasized in a plenary lecture by Kuhn in ISMP2015 that, in real life applications, determining uncertainty sets precisely or determining safe operation probabilities accurately is at least very challenging.

Another main criticism of classical robust optimization is that
satisfying all of the constraints, if not make the problem infeasible, may
lead to an objective value very far from the optimal value of the nominal problem. This
issue is more critical for large deviations.
As an example, \cite{jadid,somayeh} considered some of the problems in the NETLIB library
(under reasonable assumptions on uncertainty of certain entries) and showed that
classical robust counterparts of most of the problems in NETLIB become infeasible for a small
perturbation. Moreover, in many other problems, objective value of the classical robust
optimal solution is very low and may be unsatisfactory for the decision
maker.

Several modifications of classical robust optimization have been introduced
to deal with this issue. One, for example, is \emph{globalized robust conterparts}
introduced in Section $3$ of \cite{robust-book}. The idea is to
consider some constraints as ``soft" whose violation can be tolerated
to some degree. In this method, we take care of what happens when the data leaves
 the nominal uncertainty set. In other words, we have ``controlled deterioration"
of the constraint. These modified approaches have more flexibility than the
classical robust methodology, but we have the problem that the modified robust
counterpart of uncertain problems may become computationally
intractable.
Although the modified robust optimization framework rectifies this
drawback to some extent, it intensifies the first criticism by putting
more pressure on the DM to specify
deterministic uncertainty sets before solving the problem.

Another criticism of classical robust optimization
is that it gives the same ``weight" to all the constraints.
In practice, this is not the case as some constraints may be more
important for the DM. There are some options in classical robust
optimization like changing the uncertainty set which again
intensifies the first criticism. We will see that our
approach can alleviate these difficulties.
\subsection{Contributions and overview of this paper}
We present a framework which allows a fine-tuning of the classical tradeoff between
robustness and conservativeness by the DM and engages DM continuously and in a more 
effective way throughout the optimization process. Under a suitable theoretical modeling 
setup, we prove that the classical robust optimization approach is a special case of our
framework. We demonstrate that it is possible to efficiently perform optimization under this
framework and finally, we illustrate some of our methods in our computational experiments.

One of the main contributions of this paper is the development of
cutting-plane algorithms for robust optimization
using the notion of weighted analytic centers
in a small dimensional weight-space. We also design algorithms
in the slack variable space as a theoretical stepping stone towards the more
applicable and more efficient weight-space cutting-plane algorithms.
Ultimately, we are proposing that our approach be
used in practice with a small number (say somewhere in the order of 1 to 20)
of \emph{driving factors} that really matter to the DM. These driving factors
are independent of the number of variables and constraints, and
determine the dimension of the weight space (for interaction with the DM).
Working in a low dimensional weight-space not only simplifies the interaction
for the DM, but also makes our cutting-plane algorithms more efficient.

The notion of moving across a weight space has been widely used in the area of multi-criteria decision making: when we have several competing objective functions to optimize, a natural approach is to optimize a weighted sum of them \cite{qeval}, \cite{hu-meh}.
Authors in \cite{qeval}  presented an algorithm for evaluating and ranking items with multiple attributes. \cite{qeval} is related to our work as the proposed algorithm is a cutting-plane one. However, our algorithm uses the concept of weighted analytic center which is completely different. Authors in \cite{hu-meh} proposed a family of models (denoted by McRow) for multi-expert multi-criteria decision making.  Their work is close to ours as they derived compact formulations of the McRow model by assuming
some structure for the weight region, such as polyhedral or conic descriptions. Our work also has fundamental differences with \cite{hu-meh}: cutting-plane algorithms in the weight-space find a weight vector $w$ in a fixed weight region (the unit simplex) such that the weighted analytic center
of $w$, say $x(w)$, is the desired solution for the DM.
The algorithms we design in this paper make it possible to implement the ideas we mentioned above to help overcome some
of the difficulties for robust optimization to reach a broader, practicing user base.
For some further details and related discussion, also see Moazeni \cite{somayeh} and Karimi \cite{Karimi}.

We introduce our formulation and the notations we use in the paper in Section \ref{notions}. In Section \ref{utility}, we explain our approach and prove that, under a theoretical framework due to Bertsimas and Sim \cite{price}, our approach is as least as strong as classical robust optimization. In this section, we also introduce the notion of weighted analytic centers.  In Section \ref{alg}, we design the cutting-plane algorithms, and explain some practical uses of our approach. Some preliminary computational results are presented in Section \ref{num}. In Section \ref{con}, we briefly talk about the extension of the approach to semidefinite programming and quasi-concave utility functions, and then conclude the paper.
\section{Formulation, notations, and assumptions} \label{notions}
Before introducing our approach in the next section, let us first explain some
of the assumptions and notations we are going to use.
Much of the prior work on robust linear programming addresses the
uncertainty through the coefficient matrix. Bertsimas and Sim
\cite{bsdis} considered linear programming problems in which all
data except the right-hand-side (RHS) vector is uncertain. In
\cite{rco,uncertain,new}, it is assumed that the uncertainty affects
the coefficient matrix and the RHS vector. Some papers
deal with uncertainty only in the coefficient matrix
\cite{jadid,price,norm}. Optimization problems in which all of the data
in the objective function, RHS vector and the
coefficient matrix are subject to uncertainty, have been considered
in \cite{adjust}. As we explain
in Section \ref{utility}, the nominal data and a rough outer approximation of the uncertainty set
are enough for our approach.  However, the structure of  uncertainty region is useful for the probability analysis.
In this paper, we deal with the general setup that any part of the data $(A,b,c)$ may be subject to uncertainty; however,
we handle the uncertainty in $A$ and $c$ by first pushing the objective function into the constraints, then in the new formulation (without an objective function), by pushing
all uncertainty into the RHS. Moreover, in at least some applications, 
the amount of uncertainty in $A$ is limited whereas the uncertainty in the RHS and the objective function vectors may be 
very significant. Some of the supporting arguments for this viewpoint are:
\begin{enumerate}
\item{Instead of specifying uncertainty for each local variable, we can handle the uncertainties by lumping them into some global variables. These global variables can be, for example,
the whole budget, human resources, availability of certain critical raw materials, government quotas, etc. It may be easier for the DM to specify the uncertainty
set for these global variables.
Then, we can approximate the uncertainty in the
coefficient matrix with the uncertainty in the RHS and the objective function.
In other words, we may fix the coefficient matrix on one of the samples from
the uncertainty set and then handle the uncertainty by introducing uncertainty
to the RHS vector as in \cite{Bert-inequal}.}
\item{A certain coefficient matrix  is typical for many real world problems.
In many applications of planning and network design
problems such as scheduling, manufacturing, electric utilities,
telecommunications, inventory management and transportation,
uncertainty might only  affect costs (coefficients of the objective function) and demands (the RHS vector)\cite{net1,net2,R-GO-4}. 
\textbf{Transportation systems:}
in some problems, the nodes and the arcs are fixed. However, the cost associated to each
arc is not known precisely.
\textbf{Traffic assignment problems:} in most models, we may assume that the
drivers have perfect information about the arcs and nodes. However, their route choice behavior
makes the travelling time uncertain.
 \textbf{Distribution systems:} in some applications, the locations
of warehouses and their
capacities (in inventory planning and distribution problems) are
well-known and fixed for the DM. However, the size of
orders and the demand rate for an item could translate to an
uncertain RHS vector. Holding costs, set up costs
and shortage costs, which affect the optimal inventory cost,
are also typically uncertain. These affect at least the objective function.
\textbf{Medical/health applications:} in these applications (see for instance, 
\cite{CZHS2006,BCTT2008,SEP2012,CM2013}) the DM may be a group of people
(including medical doctors and a patient) who are more comfortable with
a few, say 4-20, driving factors which may be more easily handled by the mathematical
model, if these factors could be represented as uncertain RHS values.

In the aforementioned applications, well-understood existing resources,
reliable structures (well-established street and road networks, warehouses, and machines
which are not going to change), and logical components of the formulation are translated into a certain coefficient matrix. The
data in the objective function and the RHS vector are
usually estimated by statistical techniques by the DM, or affected by
uncertain elements such as institutional, social, or economic
market conditions. Therefore, determining
these coefficients with precision is often difficult or practically
impossible. Hence, considering uncertainty in the objective function and the
RHS vector seems to be very applicable, and motivates
us to consider such formulation in LP problems separately.}
\item  Uncertainty when restricted to the RHS and the objective function is easier to handle mathematically, in probabilistic analyses as well as in sensitivity analyses. 
\end{enumerate}
 As explained above, we represent the objective function as a constraint  $\langle c,x \rangle \geq v$, where $v$
is a lower bound specified by the information from the DM.
For example, if the DM decides that the objective value must
not be below a certain amount, we can put $v$ equal to that value. Therefore, as the input for our approach, we need the nominal values of $A$ and $c$ that we denote by $A^{(0)}$ and $c^{(0)}$, largest realizable values of $b_i$s collected in $b^{(0)}$, and a lower bound $v$ on the objective value. 
In this paper, we prefer to work only with the slack variables.
For any feasible point $x$, we define the slack variables $\bar s=b^{(0)}-A^{(0)}x$ and $s_0=-v+\langle c^{(0)},x \rangle$.
The following LP program extracted from \eqref{classicp} is the 
framework of our algorithms:
\begin{eqnarray} \label{mainp-n-2}
\ \max    && s_{0} \\
\nonumber                    \text{s.t.}    && \langle -c^{(0)},x \rangle +s_{0}= -v, \\
\nonumber                    && A^{(0)}x + \bar s =  b^{(0)},\\
\nonumber                     &&  s:= ( s_{0}, \bar s)^\top  \geq 0.
\end{eqnarray}
Our algorithms are designed for the feasible region of \eqref{mainp-n-2} and we do not need the information about the uncertainty sets. However, in Appendix \ref{prob} where we show how the uncertainty sets can affect our solutions, we assume that all the uncertainty has been pushed into the RHS. 
In view of \eqref{mainp-n-2}, let us define
\[
A:= \left [ \begin{array}{c} -c^{(0)} \\ A^{(0)}  \end{array} \right ], \ \ b:= \left [ \begin{array}{c} -v \\ b^{(0)}  \end{array} \right ].
\]
From now on, we may assume that $A \in \mathbb R^{m \times n}$ and $b  \in \mathbb R^{m}$. 
Here, without loss of generality, we impose the following restrictions on the problem (for details, see \cite{somayeh}):
The matrix $A$ has full column rank, i.e., $\textup{rank}(A)=n\leq m$. The set $\{x\in \mathbb R^{n}: Ax \leq
b\}$ is bounded and has nonempty interior.
In this paper, vectors and matrices are denoted, respectively, by lower
and uppercase letters. The matrices $Y$ and $S$ represent diagonal
matrices, having the components of vectors $y$ and $s$ on their main
diagonals, respectively. The letters $e$ and $e_{i}$ denote a vector
of all ones and the $i$th
unit basis vector with the appropriate dimension, respectively. The rows of a
matrix are shown by superscripts of the row, i.e., $a^{(i)}$ is the
$i$-th row of the matrix $A$. The inner product of two vectors $a,b \in \mathbb R^n$ is
shown both by $\langle a,b\rangle$ and $a^{\top}b$. For a matrix $A$, we show the \emph{range} of $A$
with $\mathcal R(A)$ and the \emph{null space} of $A$ with $\mathcal N(A)$.

In the next section, we introduce our utility theory based approach and compare it to classical robust optimization. In order to use robust optimization efficiently, a tractable robust counterpart is needed for a problem with uncertainty. We introduce a general framework that covers many interesting robust counterparts in the literature, and then prove two theorems that show our approach is at least as general as this framework for classical robust optimization.
%

\section{A utility theory based interactive approach and weighted analytic centers} \label{utility}

\subsection{A utility theory based interactive approach}
Consider $A$ and $b$ defined in Section \ref{notions}.
Let us define $B_s :=\{b-Ax: \ Ax \leq b\}$ as the set of all feasible slack vectors. Then we can write \eqref{mainp-n-2} as
\begin{eqnarray} \label{main p-1}
\max && \ \ U(s) \nonumber \\
s.t. && \ s \in B_s,
\end{eqnarray}
where $U(s) := s_{0}$. This $U(s)$, which we denote as utility function,  is the simplest one that takes into account only maximizing the objective function. Intuitively, we can cover a huge class of problems by using more complicated utility functions in problem \eqref{main p-1}. In this paper, we try to solve \eqref{main p-1} for a general utility function \mbox{$U: \mathbb R^{m} \rightarrow \mathbb R $} 
that models all the preferences of the DM. We do \emph{not} have access to this utility function, however assume that, for a slack vector $s$,
we can ask the DM questions to extract some information about the function. In many applications, robustness of a solution may be a monotone function of the slack variables (this typically corresponds to quasi-concave utility function in our theoretical development); however, this kind of property of the utility function is not as restrictive in our approach as it may seem since we can also handle quasi-concave utility functions. We can also use modeling techniques from goal programming (see \cite{Ignizio1976}).  Assuming that $U(s)$ is concave or quasi-concave, we retrieve the supergradient of $U(s)$ at some points through a sequence of simple questions such as pairwise comparison questions  (see for instance
\cite{KeeneyRaiffa1976,Keeney1992,KWZ2012}).

Table \ref{table-compare} compares our utility theory based interactive approach with classical robust optimization on the input to the algorithm, interaction with the DM, and handling large scale problems.  
Note that our approach is different from the heavily studied \emph{Reinforcement Learning} \cite{reinforce-1,reinforce-2}. Reinforcement Learning is a method of using 
statistical techniques and dynamic programming to estimate an explicit utility function, whereas in our approach, we do not need an explicit formulation or even an estimate. 
Our
interactive approach has the additional
benefit that in case the DM is inconsistent in his/her answers, since
our approach is interactive and operates
with very local information, we can provide the DM with a better chance
of correcting mistakes as well as
\emph{learning} throughout the interactive process, what is possible
within the given constraints and preferences. 

\newcolumntype{P}[1]{>{\centering\arraybackslash}p{#1}}
\newcolumntype{M}[1]{>{\centering\arraybackslash}m{#1}}
\renewcommand{\arraystretch}{1.5}
\begin{center}
\begin{table}[h!]
\caption{Classical robust optimization versus our utility theory based approach.}
\label{table-compare}
\begin{tabular} { >\centering m{4cm}|m{6cm}|m{6cm} }
                     &    {\bf Classical Robust Opt.}    &   {\bf  Utility Theory Based Interactive Approach } \\
     \hline 
{\bf   Input }  &    Nominal values of $A$, $b$, and $c$. \newline  Uncertainty regions for $A$, $b$, and $c$, e.g., high-dimensional ellipsoids and/or intervals \tablefootnote{We quote an axiom for robust optimization from the book \cite{robust-book}: ``The decision maker is fully responsible for consequences of the
	decisions to be made when, and only when, the actual data is
	within the prespecified uncertainty set".  }.  & Nominal values of $A$ and $c$, a lower bound on the objective function $v$, and a suitable \tablefootnote{see problem \eqref{com-1} and the explanation before equation \eqref{mainp-n-2}.} vector $b^{(0)}$. \\
     \hline
  {\bf  Communication with the DM }   &   Once at the modeling phase in receiving uncertainty regions. Once at the end, delivering the robust optimal solution.    &   Interactive throughout the whole optimization process. \\
  
 \hline 

 \multirow{2}{3cm}{{\bf \\ Handling large scale problems } }  & Large scale optimization techniques can be used for the robust counterpart.   &   In addition to large scale optimization techniques, a driving factor idea can be used to drastically reduce the dimension of search space and communication space for the DM.  \\
                    &    Specifying the uncertainty region at one shot becomes even harder as size grows.  & Connection between this small problem and the original problem requires an expert.  
\end{tabular}
\end{table}
\end{center}
In the rest of this subsection, we prove that, under a general theoretical framework, the solutions
generated by our algorithms are at least as desirable to the DM
as the solutions generated by many other robust optimization algorithms. The solution that a robust optimization technique returns is an optimal solution of a tractable robust counterpart for the LP problem with uncertainty. In the first theorem, we prove that given any optimal solution $x^*$ of a classical robust optimization problem, there exists a concave utility function $U$ such that the problem
\begin{eqnarray} \label{com-0}
\max  && g(x) := U(b-Ax) \nonumber \\
\text{s.t.}  && a_i^{\top}x \leq b_i, \,\,\, i\in \{1,\ldots,m\}.
\end{eqnarray}
has a unique optimal solution $x^*$.  We emphasize again that problem \eqref{main p-1} is not explicitly available. Next, we prove that a proper utility function always exists. 

\noindent Many classical robust optimization models and their approximations can be written as follows
\begin{eqnarray} \label{com-1}
\ \max    && c^{\top}x \\
\nonumber                    \text{s.t.}  && a_i^{\top}x + f_i(x) \leq b_i, \,\,\, i\in \{1,\ldots,m\},
\end{eqnarray}
where $f_i(x)$, $i\in \{1,\ldots,m\}$, is a convex function such that $f_i(x) \geq 0$ for all feasible $x$. If the convex uncertainty set ${\mathcal A}_i$ is known for each $i\in \{1,\ldots,m\}$ and $a_i \in {\mathcal A}_i$, then we have \mbox{$f_i(x):= \sup_{\tilde a \in {\mathcal A}_i} \tilde a ^\top x - a_i^\top x$}. By changing $f_i(x)$, different formulations can be derived.
In the following we mention some examples. Assume that for each entry $A_{ij}$ of matrix $A$ we have $A_{ij} \in [a_{ij}-\hat{a}_{ij},a_{ij}+\hat{a}_{ij}]$. It can easily be seen \cite{price} that the classical robust optimization problem is equivalent to \eqref{com-1} for $f_i(x)=\hat{a}_i^{\top} |x|$.
For the second example, assume that $A \in \{A \ : \, \left\| M(\textup{vec}(A)-\textup{vec}(\bar{A})) \right\| \leq \Delta \}$ for a given $\bar{A}$ where $\left\| . \right\|$ is a general norm and $M$ is an invertible matrix. $\textup{vec}(A)$ is a vector in $\mathbb R ^{mn \times 1}$ created by stacking the columns of $A$ on top of one another.  It is proved in \cite{norm} that many approximate robust optimization models can be formulated by changing the norm. It is also proved in \cite{norm} that this robust optimization model can be formulated as \eqref{com-1} by $f_i(x)=\Delta \left\| M^{-T} x_i \right\| _{*}$, where $\left\| . \right\|_*$ is the dual norm and $x_i \in \mathbb R ^{mn \times 1}$ is a vector that contains $x$ in entries $(i-1)n+1$ through
$in$, and zero everywhere else. 

\noindent Now, utilizing Karush-Kuhn-Tucker (KKT) theorem, we establish the following theorem.

\begin{theorem}   \label{thm::compare-1}
Assume that \eqref{com-1} has Slater points. Then, for every optimal solution $x^*$ of \eqref{com-1}, there exists
a concave function $g(x)$ (or equivalently $U(s)$) such that $x^*$ is the unique solution of \eqref{com-0}.
\end{theorem}
\begin{proof}
 For the optimality condition of \eqref{com-1} we have: There exists $\lambda \in \mathbb R^m_+$ such that
\begin{eqnarray} \label{com-2}
&&c-\sum_{i=1}^m \lambda_i (a_i+\nabla f_i(x))=0  \nonumber \\
&&\lambda_i(a_i^{\top}x + f_i(x)-b_i)=0,  \ \ \ i\in \{1,\ldots,m\}.
\end{eqnarray}
Since the Slater condition holds for \eqref{com-1}, optimality conditions \eqref{com-2} are necessary and sufficient. Let $x^*$ be an optimal solution of \eqref{com-1}, and let $J\subseteq \{1,\ldots,m\}$ denote the set of indices for which
$\lambda_i \neq 0, i \in J$. Let us define $g(x)$ as follows:
\begin{eqnarray} \label{com-3-2}
g(x):=c^{\top}x+ \sum_{i\in J} \mu_i \ln(b_i+t_i-a_i^{\top}x - f_i(x)),
\end{eqnarray}
where $t_i>0, i \in J$, are arbitrary numbers. We claim that $g(x)$ is concave. $b_i+t_i-a_i^{\top}x - f_i(x)$ is a concave function and $\ln(x)$ is increasing concave, hence $\ln(b_i+t_i-a_i^{\top}x - f_i(x))$ is a concave function for $ i\in \{1,\ldots,m\}$. $g(x)$ is the summation of an affine function and some concave functions and so is concave. The gradient of $g(x)$ is
\begin{eqnarray} \label{com-3-3}
\nabla g(x)= c- \sum_{i\in J} \frac{\mu_i}{b_i+t_i-a_i^{\top}x - f_i(x)} (a_i+\nabla f_i(x)).
\end{eqnarray}
Now define $\mu_i,  i \in J$, as
\begin{eqnarray} \label{com-4-2}
\mu_i := \lambda_i \left[b_i +t_i-a_i^{\top}x^* - f_i(x^*)\right].
\end{eqnarray}
Using \eqref{com-4-2} and comparing \eqref{com-3-3} and \eqref{com-2}, we conclude that $x^*$ is a solution of \eqref{com-0}, as we wanted. Uniqueness is because $g(x)$ is strictly convex. 
\end{proof}
The above argument proves the existence of a suitable utility function. A related, purely theoretical question is that can we construct such a utility function without having a solution of \eqref{com-2}? In the following, we construct a function with objective value arbitrarily close to the objective value of \eqref{com-1}. Assume that strong duality holds for \eqref{com-1}. Let us define
$g(x):=c^{\top}x+ \mu \sum_{i=1}^m \ln(b_i-a_i^{\top}x - f_i(x))$ and assume that $\hat{x}$ is the maximizer of $g(x)$. We have
\begin{eqnarray} \label{com-5}
\nabla g(\hat{x})= c-  \sum_{i=1}^m \frac{\mu}{b_i-a_i^{\top} \hat{x} - f_i(\hat{x})} (a_i+\nabla f_i(\hat{x})) =0.
\end{eqnarray}
This means that $\hat{x}$ is the maximizer of the Lagrangian of the problem in \eqref{com-1}, $L(\lambda,x)$, for \\
\mbox{$\hat{\lambda}_i:=\mu / (b_i-a_i^{\top} \hat{x} - f_i(\hat{x}))$}, $i \in \{1,\ldots,m\}$. So by strong duality, we have
\begin{eqnarray} \label{com-6}
c^{\top}x^* \leq L(\hat{\lambda},\hat{x}) &=& c^{\top}\hat{x} + \sum _{i=1}^m \frac{\mu}{b_i-a_i^{\top} \hat{x} - f_i(\hat{x})} (b_i-a_i^{\top} \hat{x} - f_i(\hat{x})) \nonumber \\
&=& c^{\top}\hat{x} + m\mu.
\end{eqnarray}
\eqref{com-6} shows that by choosing $\mu$ small enough, we can construct $g(x)$ such that the optimal objective value of \eqref{com-0} is arbitrarily close to the optimal objective value of \eqref{com-1}.

Note that many other approaches to robust optimization and decision making under uncertainty
(including the generalized robust counterpart introduced by Ben-Tal and
Nemirovski \cite{robust-book}, and the approach
of Iancu and Trichakis \cite{IancuTrichakis2013} using the notion of
pareto robust optimization) can be included as a special case of our framework.  A good starting
point to prove the existence of a utility function is to start with indicator functions of sets
encoding feasibility conditions. This approach first leads to utility functions that are not continuous;
however, as we showed above, these functions can be smoothed by use of barriers which then lead to
differentiable utility functions with desired properties.

To illustrate the above points, we can prove the stronger version (from the viewpoint of optimal solution sets) of Theorem \ref{thm::compare-1}.  
\begin{theorem} \label{thm::compare-2}
Assume that \eqref{com-1} has Slater points. Then, there exists
a concave function $g(x)$ (or equivalently $U(s)$) such that the sets of optimal solutions of \eqref{com-0} and  \eqref{com-1} are the same.
\end{theorem}
\begin{proof}
Note that because $f_i(x) \geq 0$, $i\in \{1,\ldots,m\}$, for all feasible points $x$, the feasible region of \eqref{com-1} is a subset of the feasible region of \eqref{com-0}. As the objective function of \eqref{com-1} is linear, the set of optimal solutions of \eqref{com-1}, denoted by $X_{opt}$, is a convex set. For an obvious choice, if we define the concave function
\begin{eqnarray*}
g(x):= \left \{ \begin{array}{rl}
	0   &   x \in X_{opt} \\
	-\infty  &  o.w.
\end{array} \right.,
\end{eqnarray*}
then \eqref{com-0} has the same set of optimal solutions as \eqref{com-1}. To show that there exists a continuous concave function, let us assume $X_{opt}$ is represented by the following system
\begin{eqnarray*}
X_{opt}=\left\{x \in \mathbb R^n:  \bar A x=\bar b, \ h_i(x) \leq 0, i\in \{1,\ldots,q\} \right\},
\end{eqnarray*}
where $\bar Ax=\bar b$ defines an affine subspace, and $h_i(x)$, $i\in \{1,\ldots,q\}$, are continuous convex functions. Consider the following function:
\begin{eqnarray*}
g(x):= \min \left\{ 0, -h_1(x),\cdots,-h_q(x) ,-\|\bar A x -\bar b\|^2\right\}.
\end{eqnarray*}
This function is concave because it is the minimum of concave functions. The maximum of the function is 0 and is achieved only on $X_{opt}$. Therefore,  the sets of optimal solutions of \eqref{com-0} and  \eqref{com-1} are the same.
\end{proof}
We can make a connection between the feasible slack vectors of an LP and the notion of weighted-analytic-centers. There are strong justifications for using weight space ($w$-space) instead of $s$-space that we will see when we design the algorithms. Besides, by using the notion of weighted center, we benefit from differentiability and smoothness of our functions in our formulations. Weight-space and weighted-analytic-centers approach embeds a ``highly differentiable" structure into the
algorithms. Such tools are extremely useful in both the theory and applications of optimization. In contrast, classical robust optimization and other competing techniques usually end up delivering a final solution where
differentiability cannot be expected; this happens because their potential optimal solutions located on the boundary of the domain of some of the structures defining the problem. 

\subsection{Definition of weighted center}  \label{weight-0}
For every $i\in\{1,2,\ldots,m\}$, let $\mathcal F_{i}$ be a closed
convex subset of $\mathbb R^{n}$ such that $\mathcal
F:=\bigcap_{i=1}^{m} \mathcal F_{i}$ is bounded and has nonempty
interior.\\
Let $F_{i}:\textup{int}(\mathcal F_{i})\rightarrow \mathbb R$ be a
self-concordant barrier for $\mathcal F_{i}$, $i\in\{1,2,\ldots,m\}$ (For a definition of self-concordant barrier functions see \cite{NN}).
For every $w\in \mathbb R_{++}^{m}$, we define the $w$-center of
$\mathcal F$ as
$$
\text{arg}\min\left\{\sum_{i=1}^{m}w_{i}F_{i}(x):x\in \mathcal
F\right\}.
$$
Consider the special case when each $\mathcal F_{i}$ is a closed half-space
in $\mathbb R^{n}$. Then the following result is well-known (see for example \cite{kojima-book,somayeh,Monteiro, Anderson}).
\begin{theorem}\label{w}
Suppose for every $i\in\{1,2,\ldots,m\}$, $a^{(i)}\in\mathbb
R^{n}\setminus\{0\}$ and $b_{i}\in\mathbb R$ are given such that:
$$
\mathcal F:=\left\{x\in\mathbb R^{n}: \langle a^{(i)},x\rangle \leq
b_{i}, \forall i\in\{1,2,\ldots,m\}\right\},
$$
is bounded and $\textup{int}(\mathcal F)$ is nonempty. Also, for
every $i\in\{1,2,\ldots,m\}$ define \mbox{$F_{i}(x):=-\ln(b_{i}-\langle
a^{(i)},x\rangle)$}. Then for every $w\in \mathbb R_{++}^{m}$, there
exists a unique $w$-center in the interior of $\mathcal F$, $x(w)$.
Conversely, for every $x\in \textup{int}(\mathcal F)$, there exists
some weight vector $w(x)\in\mathbb R_{++}^{m}$ such that $x$ is the
unique $w(x)$-center of $\mathcal F$.
\end{theorem}
Define the following family of convex optimization problems:
\begin{eqnarray}
\begin{array}{cc}
\min  & -\sum_{i=1}^{m}w_{i}\ln(s_{i})\\  
\text{s.t.}& Ax+s=b
\end{array}   \ \ \ \ \ \ 
\begin{array}{cc}
\min  & \langle b,y\rangle-\sum_{i=1}^{m}w_{i}\ln(y_{i})\\  
\text{s.t.}& A^{\top}y=0
\end{array}
\end{eqnarray}
For every
weight vector $w>0$, the objective functions of the above problems
are strictly convex on their domains. Moreover, the objective
function values tend to $+\infty$ along any sequence of their
interior points (strictly feasible points), converging to a
point on their respective boundary.
So, the above problems have minimizers in the interior of their
respective feasible regions. Since the objective functions are
strictly convex, the minimizers are unique. Therefore, for every
given $w>0$, the above problems have unique solutions $(x(w),s(w))$
and $y(w)$. 
If we write the optimality conditions for both problems, there exist $\bar y \in \mathbb R^m$ for the first problem and  $\bar s \in \mathbb R^m, \bar x \in \mathbb R^n$ for the second problem such that 
 \begin{eqnarray}
\begin{array}{c}
 Ax(w)+s(w)=b  \\
 A^\top \bar y=0\\
 S(w) \bar y = w, 
\end{array}   \ \ \ \ \ \ 
\begin{array}{c}
A \bar x+\bar s=b\\
 A^{\top}y(w)=0 \\
 \bar S y(w)=w. 
\end{array}
\end{eqnarray}
By the above uniqueness discussion, we must have $\bar y =y(w)$, $\bar s=s(w)$, and $\bar x=x(w)$, and these two systems are actually the same. 
These solutions can be used to define many \emph{primal-dual
weighted-central-paths} as the solution set
$\{(x(w),y(w),s(w)):w > 0\}$ of the following system of equations and
strict inequalities:
\begin{eqnarray}\label{analytic center}
&& Ax+s=b,\ s>0,\\
\nonumber && A^{\top}y=0,\\
\nonumber && Sy=w,
\end{eqnarray}
When we set $w:= te$, $t > 0$, we obtain the usual primal-dual
 central-path. Figure \ref{Fig-PD} illustrates some weighted central paths.
\begin{figure}[ht]
\centering
\includegraphics[height=3.2in, width = 3.5in]{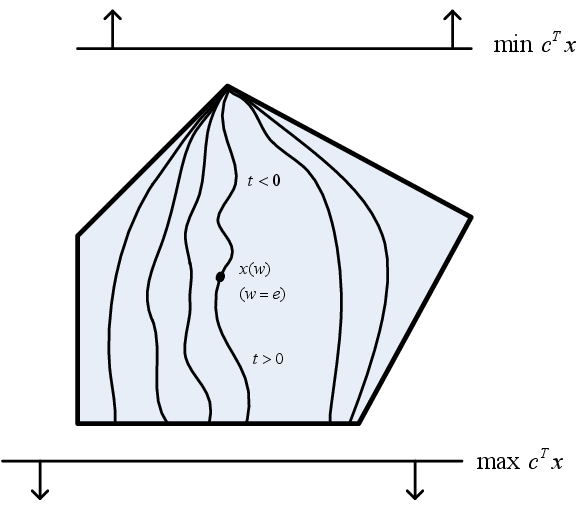}
\caption {Primal-dual central paths.}
\label{Fig-PD}
\end{figure}

For every given weight vector $w$,
$(x(w),y(w),s(w))$ is obtained uniquely from \eqref{analytic center} and $x(w)$ is called the \emph{weighted
center} of $w$. We may also refer to $(x(w),y(w),s(w))$ as the weighted
center of $w$.  For every given
$x \in \mathbb R^n$ and $y \in \mathbb R^m$, $y>0$, that satisfy the above system, $w$ and $s(w)$
are obtained uniquely. However, for a given $x \in \mathbb R^n$, there
are many weight vectors $w$ that may give $x$ as the $w$-center of the
corresponding polytope.
Without loss of generality, we restrict ourselves to the weights on the unit simplex, i.e., we consider weighted center $({x},{y},{s})$ corresponding  to weight vectors ${w}$ such that $\sum_{i=1}^m {w}_i =1$. We call this simplex of weight vectors $W$:
$$
W:=\{w \in \mathbb{R}^m : w >0, \ e^{\top}w=1\}.
$$
We will show that weight vectors in $W$ are enough to represent the feasible region (a special case can be ${w}=\frac{1}{m}e$).
We define the following notion for future reference:
\begin{definition} \label{centric}
A vector $s\in \mathbb R^m$ or $y \in \mathbb R^m$ is called \emph{centric} if there exists $x$ such that $(x,y,s)$ satisfies \eqref{analytic center} for a weight vector $w > 0$ where $e^{\top}w=1$.
\end{definition}

Let $s$ and $y$ be centric. First, we note that the simplex of the weight vectors can be divided into regions of constant $y$-vector  ($W_y$) and constant $s$-vector ($W_s$).
\begin{eqnarray}
W_y:= \{ w \in W: \ y(w)=y \}, \ \ \ W_s:=\{w \in W: \ s(w)=s\}.
\end{eqnarray}
By using Lemma \ref{W1}, if $(\hat{x},\hat{y},\hat{s})$ is the solution of system \eqref{analytic center} corresponding to the weight vector $\hat{w} \in W$, and $\bar{y}>0$ is any centric $y$-vector, then $(\hat{x},\bar{y},\hat{s})$ is the solution of system \eqref{analytic center} corresponding to the weight vector $\bar Y (\hat Y)^{-1} \hat{w}$. This means that for every centric vector $\hat{s}$ and any centric vector $y$, $\hat{S}y$ is a weight vector in the simplex.

For every pair of centric vectors $s$ and $y$, $W_s$ and $W_y$ are convex. To see this, let  $(x,\bar{y},s)$ and $(x,y,s)$ be the weighted centers of $\hat{w}$ and $w$. Then, it is easy to see that for every $\beta \in [0,1]$, $(x,\beta\bar{y}+(1-\beta) y,s)$ is the weighted center of $\beta \hat{w}+(1-\beta) w$. 
Various properties of $W_s$ and $W_y$ are studied in Appendix \ref{appendix-1}, but the following simple examples make the geometry of $W_s$ and $W_y$ clearer. We present two examples with $m=3, \ n=1$.
\begin{example} \label{Wsy}
 For the first example, let $b:=[1 \ 0 \ 0]^{\top}$ and \mbox{$A:=[1 \ -1 \ -1]^{\top}$}. By using \eqref{analytic center}, the set of centric $s$-vectors is $B_s=\{[(1-x), \ x, \ x]^{\top}: x\in(0,1)\}$. The set of centric $y$-vectors is specified by solving $A^{\top}y=0$ and $b^{\top}y=1$, while $y>0$. We can see that in this example, as shown in Figure \ref{Fig2}, $W_s$ regions are parallel line segments while $W_y$ regions are line segments which all intersect at $[1 \ 0 \ 0]^{\top}$.
\begin{figure}[ht]
\centering
\includegraphics[height=3in, width = 3.1in]{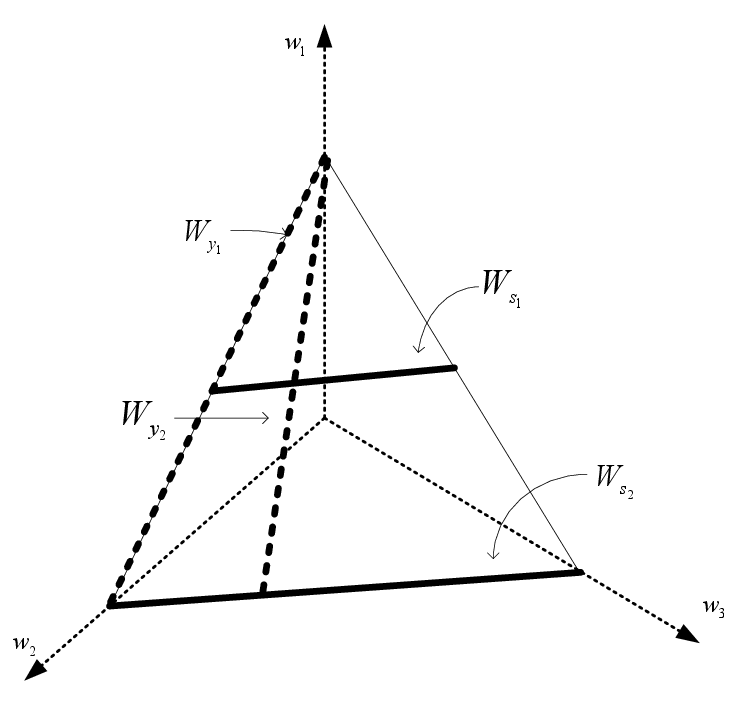}
\caption {$W_s$ and $W_y$ regions for the first example in Example \ref{Wsy}.}
\label{Fig2}
\end{figure}
For the second example, let $A:=[1 \ -1 \ 0]^{\top}$ and $b:=[1 \ 0 \ 1]^{\top}$. $W_s$ and $W_y$ regions are shown in Figure \ref{Fig3} derived by solving \eqref{analytic center}. As can be seen, this time $W_y$ regions are parallel line segments and $W_s$ regions are line segments which intersect at the point $[0 \ 0 \ 1]^{\top}$.
\begin{figure}[ht]
\centering
\includegraphics[height=3in, width = 3.1in]{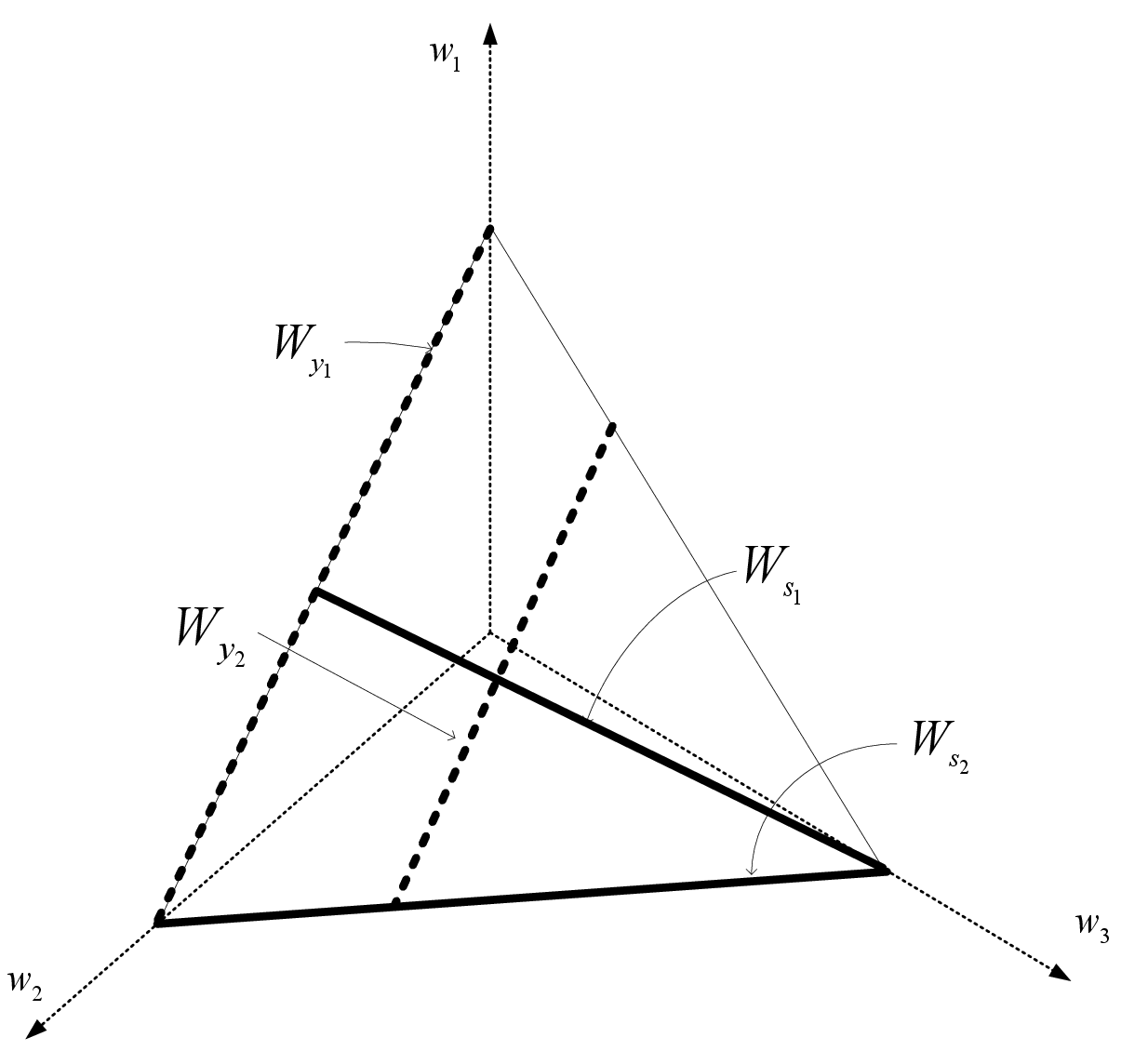}
\caption {$W_s$ and $W_y$ regions for the second example in Example \ref{Wsy}.}
\label{Fig3}
\end{figure}
\end{example}
These examples show that the affine hulls of $W_{y^1}$ and $W_{y^2}$ might not intersect for two centric $y$-vectors $y^1$ and $y^2$. This is also true for the affine hulls of $W_{s^1}$ and $W_{s^2}$ for two centric $s$-vectors $s^1$ and $s^2$.

\begin{example} \label{Wsy-2}
For the second example, let $A:=[3 \ -3 \ -2]^{\top}$ and $b:=[1 \ 1 \ 0]^{\top}$. $W_s$ and $W_y$ regions are shown in Figure \ref{Fig-lastexample}, derived by solving \eqref{analytic center}. In this example, none of
$W_y$ regions, $W_s$ regions, or their affine hulls intersect in a single point.
\begin{figure}[ht]
\centering
\includegraphics[height=3in, width = 3.1in]{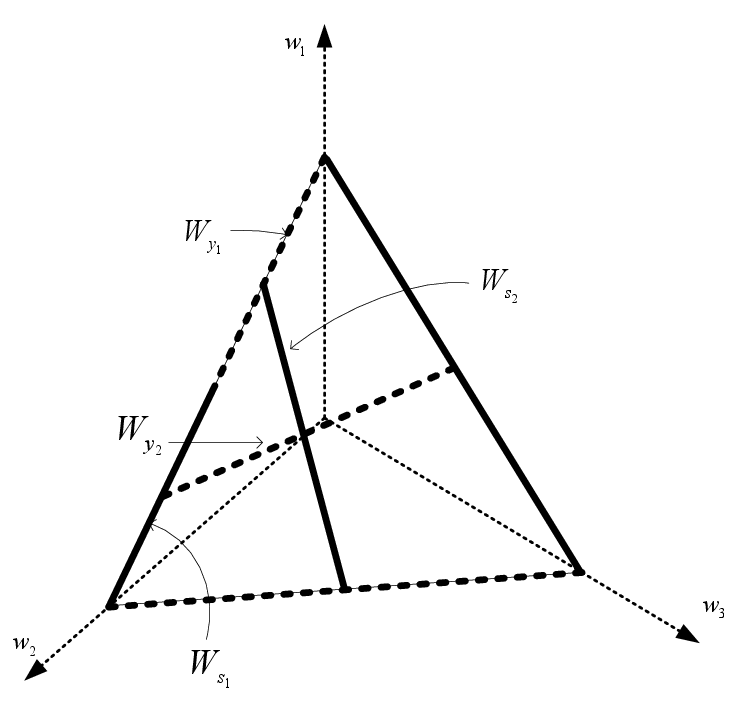}
\caption {$W_s$ and $W_y$ regions for Example \ref{Wsy-2}.}
\label{Fig-lastexample}
\end{figure}
\end{example}

\section{Algorithms} \label{alg}

In this section, we develop some cutting-plane algorithms which find an optimal solution for the DM, using the facts we established in the previous sections.
As we mentioned in Section \ref{utility},  we assume that the DM's
 preferences, knowledge, wisdom, expertise, etc. can be modeled in principle by
a utility function (as a function of the slack variables $s$),
i.e., $U(s)$, and our problem is to maximize this utility function
over the set of centric (Definition \ref{centric}) $s$-vectors $B_s$.
(Of course, we do not assume to have access to this function $U$, except
through our limited interactions with the DM.)
Therefore, our problem becomes
\begin{eqnarray} \label{main p}
\max && \ \ U(s) \nonumber \\
s.t. && \ s \in B_s.
\end{eqnarray}
In the following, we denote an optimal solution of \eqref{main p} with $s^{opt}$. In many applications, it is possible to capture choices with concave, quasi-concave, or nondecreasing utility functions. We are going to start with the assumption of concave $U(s)$. We will see in Subsection \ref{quasi-concave} that the algorithm can easily be refined to be used for quasi-concave functions. 

Suppose we start the algorithm from a point $w^0 \in \mathbb R^m$	with the corresponding $s$-vector $s^0 \in \mathbb R^m$. Using the idea of supergradient, we can introduce cuts in the $s$-space or $w$-space to shrink the set of $s$-vectors or $w$-vectors, such that the shrunken space contains an optimal point. In the following subsections, we discuss these algorithms in the $s$-space and in the $w$-space. Our main algorithm is the one in the $w$-space, however, the $s$-space algorithm helps us understand the other better.

Our algorithms are based on the notion of supergradient of a concave function. Therefore, before stating the
algorithms, we express a summary of the results we want to use.  These properties are
typically proven for convex functions in the literature \cite{rockafellar,convex-boyd}, however
we can translate all of them to concave functions.
It is a well-known fact that
for a concave function $f :  \mathbb R^n \rightarrow \mathbb R$, any local maximizer is also a global maximizer.
If a strictly concave function attains its global maximizer, it is unique. The following theorem is fundamental for developing our cutting-plane algorithms.
\begin{theorem} \label{T2}
Assume that $f :  \mathbb R^n \rightarrow \mathbb R$ is a concave function and let $x^0 \in \relint(\dom f)$. Then there exists $g \in \mathbb R^n$ such that
\begin{eqnarray} \label{super_def}
f(x) \leq f(x^0) + g^{\top}(x-x^0),  \ \ \ \forall  x \in \mathbb R^n.
\end{eqnarray}
If $f$ is differentiable at $x^0$, then $g$ is unique, and  $g=\nabla f(x^0)$.
\end{theorem}
The vector $g$ that satisfies \eqref{super_def} is called the \emph{supergradient} of $f$ at $x^0$. The set of all supergradients of $f$ at $x_0$ is called the \emph{superdifferential} of $f$ at $x^0$, and is denoted $\partial f(x^0)$.
By Theorem \ref{T2}, if $f$ is differentiable at $x^0$, then $\partial f(x^0)=\{\nabla f(x^0)\}$.
The following lemma about supergradient, which is a simple application of the chain rule, is also useful.
\begin{lemma}\label{T3}
Let $f \ :  \mathbb R^n \rightarrow \mathbb R$ be a concave function, and $D \in \mathbb R^{m \times n}$ and $b \in \mathbb R^m$ be arbitrary matrices. Then, $g(x):=f(Dx+b)$ is a concave function and we have:
$$
\partial g(x)=D^{\top} \partial f(Dx+b).
$$
\end{lemma}

\subsection{Cutting-plane algorithm in the $s$-space}

Assume that we have a starting point $s^0$ and we can obtain a supergradient $g^0$ of $U$ at $s^0$ from the DM. By using \eqref{super_def}, for every $s$,
\begin{eqnarray} \label{for-quasi}
U(s)- U(s^0) \geq 0 \ \ \Rightarrow \ \  (g^{0})^{\top}(s-s^0) \geq 0.
\end{eqnarray}
 This means that all optimal points are in the half-space $(g^{0})^{\top}(s-s^0) \geq 0$. So, by adding this cut, we can shrink the $s$-space and guarantee that there exists an optimal solution in the shrunken part. We can translate this cut to a cut in the $x$-space by using \eqref{analytic center}:
$$
(g^{0})^{\top}(s-s^0) = (g^{0})^{\top}(b-Ax-b+Ax^0) = (g^{0})^{\top}A(x^0-x).
$$
Using this equation, we consider the cut as a new constraint of the original problem; \\ \mbox{$(g^{0})^{\top}Ax \leq (g^{0})^{\top}Ax^0$}. Let us define $a^{(m+1)}=(g^{0})^{\top}A$ and $b_{m+1}=(g^{0})^{\top}Ax^0$. We redefine $\mathcal F$ by adding this new constraint and find the weighted center for a chosen weight vector $w^1$.
The step-by-step algorithm is as follows:

\textbf{$S$-space Algorithm:}
\begin{itemize}
\item{Step 1: Set $w^0=\frac{1}{m}e$ and find the $w^0$-centers $(x^0,y^0,s^0)$ with respect to ${\mathcal F}$.}
\item{Step 2: Set $k=0$, $A_0=A$, $b^0=b$, and $\mathcal F_{0}= \mathcal F$.}
\item{Step 3: If $s^k$ satisfies the DM, return $(x^k,y^k,s^k)$ and \textbf{stop}.}
\item{Step 4: Set $k=k+1$. Find $g^{k-1}$, the supergradient of $U(s)$ at $s^{k-1}$. Set
\begin{eqnarray} \label{Al1-1}
&&A_k=\left[%
\begin{array}{c}
A_{k-1} \\
(g^{k-1})^{\top}A_{k-1} \\
\end{array}
\right] ,  \ \ \ \
b^k=\left[%
\begin{array}{c}
b^{k-1} \\
(g^{k-1})^{\top}A_{k-1}x^{k-1} \\
\end{array}%
\right],  \nonumber \\
&&\mathcal F_k:=\left\{ x\in\mathbb R^{n}: \ \langle a_k^{(i)},x\rangle \leq
b^{k}_i, \forall i\in\{1,2,\ldots ,m+k\}\right\}.
\end{eqnarray}
}
\item{Step 5: Set $w_i^k=\frac{1}{m^2}$ for $i\in\{m+1,\ldots,m+k\}$ and $w_i^k=\frac{1}{m}-\frac{k}{m^2}$ for $i\in\{1,\ldots,m\}$. Find the $w^k$-center $(x^k,y^k,s^k)$ with respect to $\mathcal F_k$. Return to Step 3.}
\end{itemize}

The logic behind Step 5 is that we want to give smaller weights to the new constraints than the original ones
(however, our choices above are just examples; implementers should make suitable, practical
choices that are tailored to their specific application). A main problem with this algorithm is
that the dimension of the weight-space is very large and is increased by one every time we add a constraint.
We show that this problem is solved by our $w$-space algorithm and the notion of driving factors in the following subsections.

\subsection{Cutting-plane algorithm in the $w$-space}

\label{sec:algorithm w-space}
In this subsection, we consider the cuts in the $w$-space. To do that, we first try a natural way of extending the algorithm in the $s$-space to the one in the $w$-space. We show that this extension only works for a limited subset of utility functions. Then, we develop an algorithm applicable to all concave utility functions.

Like the $s$-space, we try to use the supergradients of $U(s)$. Let $U_w$ denote the utility function as a function of $w$. From \eqref{analytic center} we have $Ys=w$; so, $U_w(w)=U(s)=U(Y^{-1}w)$. If $Y$ were constant for all weight vectors, $U_w(w)$ would be a concave function, and we could use Lemma \ref{T3} to find the supergradient at each point.
The problem here is that $Y$ is not necessarily the same for different weight vectors. Assume that we confine ourselves to weight vectors in the simplex $W$ with the same $y$-vector ($W_y$). $U_w(w)$ is a concave function on $W_y$, so, we can define its supergradient. By Lemma \ref{T3}, we conclude that  $\partial U_w(w)= Y^{-1} \partial U(s)$ for all $w \in W_y$.

Suppose we start at $w^0$ with the weighted center $(x^0,y^0,s^0)$. Let us define $g^{0w}:=(Y^0)^{-1}g^0$, where $g^0$ is a supergradient of $U(s)$ at $s^0$. Then, from \eqref{super_def} we have,
\begin{eqnarray} \label{superg-w}
U_w(w) \leq U_w(w^0) + (g^{0w})^{\top}(w-w^0), \ \ \ \forall w \in W_{y^0}.
\end{eqnarray}


\noindent If we confine the weight-space to $W_y$, by the same procedure used for $s$-space, we can introduce cuts in the $w$-space using \eqref{superg-w}. The problem is that we do not have a proper characterization of $W_y$. On the other hand, $U_w$ may not be a concave function on the whole simplex.
Assume that $s^{opt}$ is an optimal solution of \eqref{main p}, and $W_{s^{opt}}$ is the set of weight vectors in the simplex with $s$-vector $s^{opt}$. It is easy to see that $W_{s^{opt}}$ is convex. We also have the following lemma:

\begin{lemma} \label{FL}
Let $(x',y',s')$ be the weighted center corresponding to $w'$, $s^{opt}$ be an optimal solution of \eqref{main p}, and $g'$ be the supergradient of $U(s)$ at $s'$. Then, $S^{opt}y'$ is in the half-space ${g'}_{w}^{\top}(w-w') \geq 0$, where $g'_w=Y'^{-1}g'$.
\end{lemma}
\begin{proof}
We have
${g'}_{w}^{\top}(S^{opt}y'-w')=g'^{\top}{Y'}^{-1}(S^{opt}y'-S'y')=g'^{\top}(s^{opt}-s') \geq 0.$
The last inequality follows from the fact that $s^{opt}$ is a maximizer and $g'$ is a supergradient of $U(s)$ at $s'$.
\end{proof}

The above lemma shows that using hyperplanes of the form $g'^{\top}Y'^{-1}(w-w')$, we can always keep a point from $W_{s^{opt}}$. Now, using the fact that $W_{s^{opt}}$ is convex and the above lemma, the question is: if we use a sequence of these hyperplanes, can we always keep a point from $W_{s^{opt}}$?  A simpler question is: We start with $w^0$ and shrink the simplex $W$ into the intersection of the half-space $(g^{0w})^{\top}(w-w^0) \geq 0$ and the simplex, say $W_0$. Then we choose an arbitrary weight vector $w^1$ with weighted center $(x^1,y^1,s^1)$ from the shrunken space $W_0$. If $g^1$ is a supergradient of $U(s)$ at $s^1$, then we shrink $W_0$ into the intersection of $W_0$ and the half-space $(g^{1w})^{\top}(w-w^1) \geq 0$, where $g^{1w}=(Y^{1})^{-1}g^1$, and call the last shrunken space $W_1$. Is it always true that a weight vector with $s$-vector $s^{opt}$ exists in $W_1$?
In the following, we show that this is true for some utility functions, but not true in general. We define a special set of functions that have good properties for cuts in the $w$-space, and the above algorithm works for them.
\begin{definition} \label{NDAS}
A function $f: \mathbb R^m_{++} \rightarrow \mathbb R$ is called \emph{Non-Decreasing under Affine Scaling (NDAS)} if for every $d \in \mathbb R^m_{++}$ we have:
\begin{enumerate}
\item{$
f(s) \ \leq \max\{f(Ds),f(D^{-1}s)\}, \ \ \ \forall s \in \mathbb R^m_{++}.
$}
\item{If for a single $s^0 \in \mathbb R^m_{++}$ we have $f(s^0) \leq f(Ds^0)$, then $f(s) \leq f(Ds)$ for all $s \in \mathbb R^m_{++}$.}
\end{enumerate}
\end{definition}
For every $t \in \mathbb R^m$ the function $f_1(s):=\sum_{i=1}^mt_i \log{s_i}$ is NDAS. Indeed, for every $s,d \in \mathbb R^m_{++}$ we have:
\begin{eqnarray*}
f_1(s)-f_1(Ds) &=& -\sum_{i=1}^mt_i \log{d_i}, \nonumber\\
f_1(s)-f_1(D^{-1}s) &=& -\sum_{i=1}^mt_i \log{\frac{1}{d_i}}=\sum_{i=1}^mt_i \log{d_i},
\end{eqnarray*}
and so we have $2f_1(s)=f_1(Ds)+f_1(D^{-1}s)$. The second property is also easy to verify and the function is NDAS. $f_1(s)$ is also important due to its relation to a family of classical utility functions in mathematical economics: Cobb-Douglas production function which is defined as $U_{cd}(s)=\prod_{i=1}^m s_i^{t_i}$, where $t \in \mathbb R^m_{++}$. Usage of this function to simulate problems in economics goes back to at least the 1920's. Maximization of $U_{cd}(s)$ is equivalent to the maximization of its logarithm which is equal to \\ $f_1(s)=\ln(U_{cd}(s))=\sum_{i=1}^mt_i \log{s_i}$.
Authors in \cite{qeval} considered the Cobb-Douglas utility function to present an algorithm for evaluating and ranking items with multiple attributes. \cite{qeval} is related to our work as the proposed algorithm is a cutting-plane one. \cite{qeval} also used the idea of weight-space as the utility function is the weighted sum of the attributes. However, our algorithm uses the concept of weighted analytic center which is different.
Now, we have the following proposition.
\begin{proposition} \label{W-l3}
Assume that $U(s)$ is a NDAS concave function.
Let $(x^0,y^0,s^0)$ and $(x^1,y^1,s^1)$ be the weighted centers of $w^0$ and $w^1$, and $g^0$ and $g^1$ be the supergradients of $U(s)$ at $s^0$ and $s^1$, respectively. Then we have
$$
\left \{w: \ (g^{0w})^{\top}(w-w^0) \geq 0, \ (g^{1w})^{\top}(w-w^1) \geq 0  \right \} \cap W_{s^{opt}}  \neq \ \phi,
$$
where $g^{0w}=(Y^{0})^{-1}g^0$ and $g^{1w}=(Y^{1})^{-1}g^1$.
\end{proposition}
\begin{proof}
Consider the weight vectors $Y^0s^{opt}$ and $Y^1s^{opt}$. Our two hyperplanes are
\begin{eqnarray*}
&&P_0:= \{w: \ \ (g^0)^{\top}(Y^{0})^{-1}(w-Y^0s^0)=0 \}, \\
&&P_1:= \{w: \ \ (g^1)^{\top}(Y^{1})^{-1}(w-Y^1s^1)=0 \}.
\end{eqnarray*}
By Lemma \ref{FL}, $Y^0s^{opt}$ is in the half-space $(g^0)^{\top}(Y^{0})^{-1}(w-Y^0s^0) \geq0$  and $Y^1s^{opt}$ is in the half-space $(g^1)^{\top}(Y^{1})^{-1}(w-Y^1s_1) \geq 0$. If one of these two points is also in the other half-space, then we are done. So, assume that
$$
(g^0)^{\top}(Y^{0})^{-1}(Y^1s^{opt}-Y^0s^0) < 0 \ \ \textup{and} \ \ (g^1)^{\top}(Y^{1})^{-1}(Y^0s^{opt}-Y^1s^1) < 0
$$
(we are seeking contradiction), which is equivalent to
\begin{eqnarray}\label {W-l2-1}
 (g^0)^{\top}((Y^{0})^{-1}Y^1s^{opt}-s^0) < 0 \ \ \textup{and} \ \ (g^1)^{\top}((Y^{1})^{-1}Y^0s^{opt}-s^1) < 0.
\end{eqnarray}
Using \eqref{superg-w} and \eqref{W-l2-1} we conclude that
\begin{eqnarray*}
&&U((Y^{0})^{-1}Y^1s^{opt}) < U(s^0) \leq U(s^{opt}) \ \ \textup{and} \\
&&U((Y^{1})^{-1}(Y^0)s^{opt}) < U(s^1) \leq U(s^{opt}).
\end{eqnarray*}
However, note that $(Y^{0})^{-1}Y^1=((Y^{1})^{-1}Y^0)^{-1}$ and this is a contradiction to Definition \ref{NDAS}. So \eqref{W-l2-1} is not true and at least one of $Y^0s^{opt}$ and $Y^1s^{opt}$ is in
$$\{w: \ (g^{0w})^{\top}(w-w^0) \geq 0, \ (g^{1w})^{\top}(w-w^1) \geq 0 \}.$$
\end{proof}
By Proposition \ref{W-l3}, using the first two hyperplanes, the intersection of the shrunken space and $W_{s^{opt}}$ is not empty.
Now, we want to show that we can continue shrinking the space and have nonempty intersection with  $W_{s^{opt}}$.

\begin{proposition} \label{NDAS-thm}
Assume that $U(s)$ is a NDAS concave function.
Let $(x^i,y^i,s^i)$ be the weighted centers of $w^i$, $ i \in \{0, \ldots , k\}$, and $g^i$ be the supergradients of $U(s)$ at $s^i$.
Let us define
$$
W^i:=\left \{w: \ (g^{iw})^{\top}(w-w^i) \geq 0  \right \} \cap W,
$$
where $g^{iw}=(Y^{i})^{-1}g^i$. Assume we picked the points such that
\begin{eqnarray} \label{NDAS-thm-1}
w^i \in \relint \left ( \bigcap_{j=0}^{i-1} W^j \right ), \ \ \ \ i \in \{1, \ldots , k\}.
\end{eqnarray}
Then we have
\begin{eqnarray} \label{NDAS-thm-2}
\left ( \bigcap_{j=0}^{k} W^j \right ) \ \cap \ W_{s^{opt}} \neq \ \phi,
\end{eqnarray}
where $s^{opt}$ is an optimal solution of \eqref{main p}.
\end{proposition}
\begin{proof}
Among the three representations of $W_s$ were given in \eqref{W_s}, we use the second one in the following. If \eqref{NDAS-thm-2} is not true, then the following system is infeasible:
\begin{eqnarray} \label{NDAS-thm-3}
A^{\top}(S^{opt})^{-1}w = 0, \ \ \ e^{\top}w=1,  \ \ \ w \ge 0, \nonumber \\
(g^{iw})^{\top}(w-w^i) \geq 0, \ \ \ i \in \{0, \ldots , k\}.
\end{eqnarray}
By Farkas' Lemma, there exist $v \in \mathbb R^n$, $p \in \mathbb R$, and $q \in \mathbb R^k_{+}$ such that:
\begin{eqnarray} \label{NDAS-thm-4}
&&(S^{opt})^{-1}Av+pe-\sum_{i=0}^k q_i g^{iw} \geq 0  \ \ \Leftrightarrow \ \ Av+ps^{opt}-\sum_{i=0}^k q_i S^{opt}(Y^{i})^{-1}g^i \geq 0, \nonumber \\
&&p -  \sum_{i=0}^k q_i (g^{iw})^{\top} w^i < 0 \ \ \Leftrightarrow \ \ p -  \sum_{i=0}^k q_i (g^{i})^{\top} s^i < 0.
\end{eqnarray}
Now for each $j \in \{0, \ldots , k\}$, we multiply both sides of the first inequality in \eqref{NDAS-thm-4} with $e^{\top}Y^{j}$, then we have:
\begin{eqnarray} \label{NDAS-thm-5}
&&p -\sum_{i=0}^k q_i (s^{opt})^{\top} Y^{j}(Y^{i})^{-1}g^i \geq 0,  \ \ \ \forall j \in \{0, \ldots , k\}, \nonumber \\
&&p -  \sum_{i=0}^k q_i (g^{i})^{\top} s^i < 0,
\end{eqnarray}
where we used the facts that $e^{\top}Y^{j}Av=(A^{\top}y^{j})^{\top}v=0$ and $e^{\top}Y^{j}s^{opt}=1$. If we multiply the first set of inequalities in \eqref{NDAS-thm-5} with $-1$ and add it to the second one we have
\begin{eqnarray} \label{NDAS-thm-6}
q_j(g^{j})^{\top} (s^{opt}-s^j) + \sum_{i\neq j}  q_i (g^{i})^{\top} (Y^{j}(Y^{i})^{-1}s^{opt}-s^i) < 0,
\end{eqnarray}
for all $j \in \{0, \ldots , k\}$. $q \in \mathbb R^k_{+}$ and $(g^{j})^{\top} (s^{opt}-s^j) \geq 0$ by supergradient inequality. Hence, from \eqref{NDAS-thm-6}, for each $j \in \{0, \ldots , k\}$, there exists $\phi_j \in \{0, \ldots , k\} \backslash \{j\}$ such that $(g^{\phi_j})^{\top} (Y^{j}(Y^{\phi_j})^{-1}s^{opt}-s^{\phi_j}) < 0$ which, using \eqref{for-quasi}, means $U(Y^{j}(Y^{\phi_j})^{-1}s^{opt}) < U(s^{\phi_j}) \leq U(s^{opt}) $. Therefore, by the first property of NDAS functions, we must have
\begin{eqnarray} \label{NDAS-thm-7}
U(Y^{\phi_j}(Y^{j})^{-1}s^{opt}) \geq U(s^{opt}).
\end{eqnarray}
Now, it is easy to see that there exists a sequence $j_1, \ldots, j_t \in \{0, \ldots , k\} $ such that $\phi_{j_i}=j_{i+1}$ and $\phi_{j_t}=j_{1}$. By using \eqref{NDAS-thm-7} and the second property of NDAS functions $t-1$ times we can write:
\begin{eqnarray} \label{NDAS-thm-8}
U(s^{opt}) & \leq & U(Y^{j_2}(Y^{j_1})^{-1}s^{opt}) \nonumber \\ & \leq & U(Y^{j_3}(Y^{j_2})^{-1}Y^{j_2}(Y^{j_1})^{-1}s^{opt}) \nonumber \\ &\leq& \cdots \leq U(Y^{j_t}(Y^{j_{t-1}})^{-1} \cdots Y^{j_2}(Y^{j_1})^{-1}s^{opt}) \nonumber \\ &=& U(Y^{j_t}(Y^{j_1})^{-1}s^{opt}).
\end{eqnarray}
However, we had $U(Y^{j_t}(Y^{j_1})^{-1}s^{opt})= U(Y^{j_t}(Y^{\phi_{j_t}})^{-1}s^{opt}) < U(s^{opt})$ which is a contradiction to \eqref{NDAS-thm-8}. This means the system \eqref{NDAS-thm-3} is feasible and we are done.
\end{proof}
Proposition \ref{NDAS-thm} shows that the above-mentioned cutting-plane algorithm works for the NDAS functions.
However, the conclusion of the proposition is not true for a general concave function. For a counter example, see Example \ref{CoE} in Appendix \ref{appendix}.  To be able to perform a cutting-plane algorithm in the $w$-space, we modify the definition of cutting hyperplanes. In the next two propositions, we introduce a new set of cutting-planes.

\begin{proposition} \label{W-l5}
For every point $Y^0s^0 \in W$, there exists a hyperplane $P$ passing through it such that: \newline
1- P contains all the points in $W_{s^0}$, and \newline
2- P cuts $W_{y^0}$ the same way as $(g^0)^{\top}(Y^0)^{-1}(w-Y^0s^0)=0$ cuts it; the intersections of P and $(g^0)^{\top}(Y^0)^{-1}(w-Y^0s^0)=0$ with $W_{y^0}$ is the same, and the projections of their normals onto $W_{y^0}$ have the same direction.
\end{proposition}
\begin{proof}
Assume that $w^0=Y^0s^0$ is the point that is chosen and let $u^0$ be the normal vector to the desired hyperplane $P$. First, we want the hyperplane to contain $W_{s^0}$. This means that for all centric $\hat{y}$, the vector $S^0y^0-S^0\hat{y}$ is on $P$, i.e., we have $(u^0)^{\top}S^0(y^0-\hat{y})=0$. Since $A^{\top}(y^0-\hat{y})=0$, we can put $u^0=(S^0)^{-1}Ah^0$ with an arbitrary $h^0$ and we have:
$$
(u^0)^{\top}S^0(y^0-\hat{y})=(h^0)^{\top}A^{\top}(S^0)^{-1}S^0(y^0-\hat{y})=0.
$$
Now, we want to find $h^0$ such that $(u^0)^{\top}(w-Y^0s^0)$ cuts $W_{y^0}$ the same way as \\ $(g^0)^{\top}(Y^0)^{-1}(w-Y^0s^0)$ cuts it. We actually want to find $h^0$ which satisfies the stronger property that $(u^0)^{\top}(w-Y^0s^0)=(g^0)^{\top}(Y_0)^{-1}(w-Y^0s^0)$ for all $w \in W_{y^0}$. All the points in $W_{y^0}$ are of the form $Y^0\hat{s}$, so we must have $(u^0)^{\top}Y^0(\hat{s}-s^0)=(g^0)^{\top}(\hat{s}-s^0)$. Since $(\hat{s}-s^0)$ is in the range of $A$, this equation is true for every $\hat s$ if and only if:
$$
(u^0)^{\top}Y^0Ax=(g^0)^{\top}Ax \ \Rightarrow \ ((u^0)^{\top}Y^0-(g^0)^{\top})Ax=0, \ \ \forall x \in \mathbb R^n.
$$
This means that $Y^0u^0-g^0$ must be in $\mathcal R(A)^{\bot}=\mathcal N(A^{\top})$, which means $A^{\top}(Y^0u^0-g^0)=0$. However, we had from above that $u^0=(S^0)^{-1}Ah^0$ and hence:
\begin{eqnarray} \label{normal-1}
A^{\top}Y^0u^0=A^{\top}g^0 \ \Rightarrow \ A^{\top}Y^0(S^0)^{-1}Ah^0=A^{\top}g^0 \ \Rightarrow \ h^0=(A^{\top}Y^0(S^0)^{-1}A)^{-1}A^{\top}g^0.
\end{eqnarray}
So, the hyperplane with normal vector $u^0=(S^0)^{-1}Ah^0$, where
$h^0=(A^{\top}Y^0(S^0)^{-1}A)^{-1}A^{\top}g^0$ has the required properties.
Since this hyperplane cuts $W_{y^0}$ the same way as $(g^0)^{\top}(Y^0)^{-1}(w-Y^0s^0)$ does, we conclude that $(u^0)^{\top}(Y^0s^{opt}-Y^0s^0) \geq 0$. Therefore, $Y^0s^{opt}$ is in the half-space $(u^0)^{\top}(w-Y^0s^0) \geq 0$.
\end{proof}
The normal of the hyperplane derived in Proposition \ref{W-l5} has a nice interpretation with respect to orthogonal projection and the primal-dual scaling $Y^{-1}S$. We have:
\begin{eqnarray} \label{u0-NPro}
u^0&=&(S^0)^{-1}A(A^{\top}Y^0(S^0)^{-1}A)^{-1}A^{\top}g^0 \nonumber \\
&=& (Y^0)^{-1/2}(S^0)^{-1/2} \nonumber \\ && \ \ \ \ \ \
\underbrace{[((Y^0)^{1/2}(S^0)^{-1/2}A) (A^{\top}Y^0(S^0)^{-1}A)^{-1} (A^{\top}(S^0)^{-1/2}(Y^0)^{1/2})]}_{\Pi} (Y^0)^{-1/2}(S^0)^{1/2}g_0 \nonumber \\
&=& (Y^0)^{-1/2}(S^0)^{-1/2} P (Y^0)^{-1/2}(S^0)^{1/2}g_0,
\end{eqnarray}
where $\Pi$ is the orthogonal projection onto the range of $(Y^0)^{1/2}(S^0)^{-1/2}A$.
Note that a main benefit of the hyperplane in Proposition \ref{W-l5} is that when we choose a point $Y^0s^0$, we can cut away $W_{s^0}$. Now, we prove the following proposition which shows we can cut the simplex with a sequence of hyperplanes such that the intersection of their corresponding half-spaces contain a point from $W_{s^{opt}}$.
\begin{proposition} \label{W-l6}
Assume that we choose the points $Y^0s^0, Y^1s^1 \in W$. The hyperplane $P$ passing through $Y^1s^1$, with the normal vector $u^1:=(S^1)^{-1}Ah^1$ , $h^1=(A^{\top}Y^0(S^1)^{-1}A)^{-1}A^{\top}g^1$ satisfies the following properties: \newline
1- P contains all the points in $W_{s^1}$, and \newline
2- $(u^1)^{\top}(Y^0s^{opt}-Y^1s^1) \geq 0$ for every feasible maximizer of $U(s)$.
\end{proposition}
\begin{proof}
As in the proof of Proposition \ref{W-l5}, if we set $u^1=(S^1)^{-1}Ah^1$, then the hyperplane contains all the points in $W_{s^1}$.
To satisfy the second property, we want to find $h^1$ with the stronger property that
\begin{eqnarray} \label{cuteq1}
(u^1)^{\top}(Y^0\hat{s}-Y^1s^1)=(g^1)^{\top}(\hat{s}-s^1),
\end{eqnarray}
for all the centric $\hat{s}$. The reason is that we already have $(g^1)^{\top}(s_{opt}-s^1) \geq 0$. By the choice of $u^1=(S^1)^{-1}Ah^1$, for every centric $y$ we have $$(u^1)^{\top}S^1y=(h^1)^{\top}A^{\top}(S^1)^{-1}S^1y=(h^1)^{\top}A^{\top}y=0.$$ So, we have $(u^1)^{\top}Y^1s^1=(u^1)^{\top}Y^0s^1=0$ and we can continue the above equation as follows:
\begin{eqnarray*}
(g^1)^{\top}(\hat{s}-s^1)&=&(u^1)^{\top}(Y^0\hat{s}-Y^1s^1)=(u^1)^{\top}(Y^0\hat{s}) \nonumber\\
&=&(u^1)^{\top}(Y^0\hat{s}-Y^0s^1) \nonumber\\
&=&(u^1)^{\top}Y^0(\hat{s}-s^1).
\end{eqnarray*}
Now we can continue in a similar way as in the proof of Proposition \ref{W-l5}. Since $(\hat{s}-s^0)$ is in the range of $A$, we must have:
$$
((u^1)^{\top}Y^0-(g^1)^{\top})Ax=0, \ \ \forall x \in \mathbb R^n.
$$
By the same reasoning, we have:
\begin{eqnarray} \label{normal-2}
A^{\top}Y^0u^1=A^{\top}g^1 \ \Rightarrow \ A^{\top}Y^0(S^1)^{-1}Ah^1=A^{\top}g^1 \ \Rightarrow \ h^1=(A^{\top}Y^0(S^1)^{-1}A)^{-1}A^{\top}g^1.
\end{eqnarray}
So, the hyperplane with normal vector $u^1=(S^1)^{-1}Ah^1$, where $h^1=(A^{\top}Y^0(S^1)^{-1}A)^{-1}A^{\top}g^1$ has the required properties.
\end{proof}
By Proposition \ref{W-l6}, we can create a sequence of points and hyperplanes such that the corresponding half-spaces contain $Y^0s^{opt}$.
The algorithm is as follows:

\textbf{$W$-space Algorithm:}
\begin{itemize}
\item{Step 1: Set $w^0=\frac{1}{m}e$ and find the $w^0$-centers $(x^0,y^0,s^0)$ with respect to ${\mathcal F}$.}

\item{Step 2: Set $k=0$, and $W_0=W$.}
\item{Step 3: If $s^k$ satisfies the optimality condition, return $(x^k,y^k,s^k)$ and \textbf{stop}.}
\item{Step 4: Find $g^{k}$, the supergradient of $U(s)$ at $s^{k}$. Find $h^{k}$ by solving the following
equation
\begin{eqnarray}
A^{\top}Y^0(S^{k})^{-1}Ah^{k}=A^{\top}g^{k}.
\end{eqnarray}
}
\item{Step 5: Set $u^{k}=(S^{k})^{-1}Ah^k$ and $W_{k+1}=W_{k} \cap \{ w: \ (u^{k})^{\top}(w-w^k) \geq 0 \}$. Pick a point $w^{k+1}$ from $W_{k+1}$ (see subsection \ref{convergence}) and find the $w^{k+1}$-center $(x^{k+1},y^{k+1},s^{k+1})$ with respect to $\mathcal F$. Set $k=k+1$ and return to Step 3.}
\end{itemize}

A clear advantage of this algorithm over the one in the $s$-space is that we do not have to increase the dimension of the $w$-space at each pass and subsequently we do not have to assign weights to the newly added constraints. So, the above algorithm is straightforward to implement. The convergence of the algorithm depends on the choice of $w^{k+1}$ in Step 5, which we discuss in Subsection \ref{convergence}. We can also use  the properties of the weighted center we derived in Appendix \ref{appendix-1} to improve the performance of the algorithm. 

\subsection{Some implementation ideas}
In the previous subsections, we introduced an algorithm that is highly cooperative with the DM and proved many interesting features about it. In this subsection, we set forth some implementation ideas.
\subsubsection{Driving factors}
As we mentioned, one of our main criticisms of classical robust optimization is that it is not practical to ask the DM to specify an $m$-dimensional ellipsoid for the uncertainty set. Our approach improves this situation by asking easier questions. 
The idea is similar to those used in the area of multi-criteria optimization. Consider the system of inequalities $Ax \leq b$ and the corresponding slack vector $s=b-Ax$ representing the problem. The DM might prefer to directly consider
only a few factors that really matter, we call them \emph{Driving Factors}. For example, the driving factors for a DM might be budget amount, profit, allocated human resources, etc. We can represent $k$ driving factors by $(c^i)^{\top}x$, $i\in \{1,\ldots,k\}$, and the problem for the DM is to maximize the utility function $U((c^1)^{\top}x,\ldots,(c^k)^{\top}x)$. Similar to the way we added the objective of the linear program to the constraints, we can add $k$ constraints to problem and write \eqref{main p} as:
\begin{eqnarray} \label{main p-m}
\max &&  U(\xi_1,\ldots,\xi_k) \nonumber \\
s.t. && \xi_i=\hat{b}_i-(c^i)^{\top}x, \ \  \xi \geq 0, \ \ i\in \{1,\ldots,k\} \nonumber \\
 &&     s= b-Ax, \ \ s \geq 0. 
\end{eqnarray}
As can be seen, the supergradient vector has only $k$ nonzero elements which makes it much easier for the DM to specify it for $k \ll m$.
$k$ is usually very small and we can figure out approximate gradients by asking pair-wise comparison questions among the driving factors. However, it still may have the problem that 
the cutting plane algorithm is in a high-dimensional space and it might be slow. We can take one step further to resolve this difficulty. 

\noindent Consider the following setup: 
\begin{itemize}
\item A very large  system of equalities and inequalities $Ax+s=b$, $s \geq 0$. 
\item A very small driving factor system in the space of $\xi$ `variables. Our goal is to solve problem  
\begin{eqnarray} \label{main p-m-2}
\max && U(\xi_1,\ldots,\xi_k) \nonumber \\
s.t. && \xi \in B_{\xi},
\end{eqnarray}
where 
$B_\xi:=\left\{ \xi \in \mathbb R^k_+: \xi_i=\hat{b}_i-(c^i)^{\top}x, i\in \{1,\ldots,k\}, Ax \leq b \right\}$
and, without loss of generality, we may assume that $ U(\xi_1,\ldots,\xi_k)$ is a monotone non-decreasing function of $\xi_1, \ldots, \xi_k$.
\item A matrix $C$ and a vector $d$ such that $\xi = C s+d$. 
\item A matrix $\bar C$ that translates a displacement in the driving factor space $d_{\xi}$ to a displacement in the $s$-space $d_s$, i.e., $d_s=\bar C d_\xi$. Note that considering $d_\xi=Cd_s$, there are infinite number of choices for $\bar C$ and finding the most effective one can be done by an optimizer/expert. The system $d_s=\bar C d_\xi$ actually showing how to change big space variables when there is a change in the driving factors. For example, assume that $d_\xi$ requires decreasing workforce in a retail corporation with several branches. The change in the workforce for each individual branch, embedded in $d_s$, should be done by an optimizer/expert. However, to mention one possibility, we may consider the pseudoinverse of $C$ as $\bar C:= C^\top (CC^\top)^{-1}$.
\end{itemize}
\eqref{main p-m-2} is a problem in a $k$-dimensional space (say, $k \in \{1,2,\ldots,20\}$)  and can be solved efficiently 
with our cutting-plane algorithms. Assume that at the $k$'th iteration we have a feasible slack vector $s^k$ in the big space and a feasible slack vector $\xi^k$ in the very small driving factor space, and by applying our algorithm in the driving factor space, we get a search direction $d_\xi$. Using $\bar C$ we get $d_s:=\bar Cd_\xi$ to update $s^{k+1}=s^k+\alpha d_s$ for an appropriate $\alpha>0$. Algorithm in the $w_\xi$-space stops quickly, and we have a good estimate of the optimal weights in $w$-space.

The DM deals only with problem \eqref{main p-m-2} directly, however, an optimizer/expert needs to translate the cuts (and the information extracted from the DM) in $w_\xi$-space into changes in the current assignment of weights in the big $w$-space and coordinate the search between the $w_\xi$-space and $w$-space (see Figure \ref{Fig_DF1}). 
\begin{figure}[ht]
\centering
\includegraphics[height=1.5in, width = 6in]{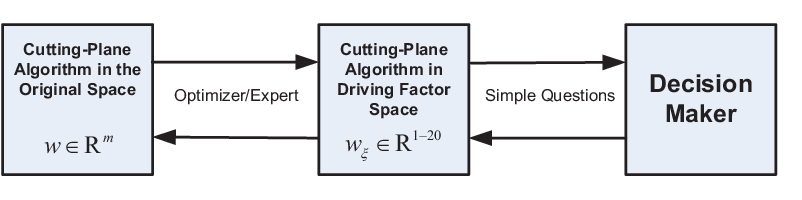}
\caption {Diagram for the driving factor approach.}
\label{Fig_DF1}
\end{figure}

\subsubsection{Approximate gradients}

In the previous subsection, we derived a cutting-plane algorithm in the $w$-space. As can be seen from Propositions \ref{W-l5} and \ref{W-l6}, for the implementation we need the supergradients of the utility function $U(s)$. However, we usually do not have an explicit formula for $U(s)$ and our knowledge about it comes from the interaction with the DM.
Supplying supergradient information on preferences (i.e., the utility function)
might still be a difficult task for the DM.
So, we have to simplify our questions for the DM and
try to adapt our algorithm accordingly.

We try to derive approximate supergradients based on simple questions from the DM. The idea is similar to the one used by Arbel and Oren in \cite{Ardel-Oren}. Assume that $U(s)$ is differentiable which means the supergradient at each point is unique and equal to the gradient of the function at that point. Assume that the algorithm is at the point $s$. By Taylor's Theorem (first order expansion) for arbitrarily small scalars $\epsilon_i >0$ we have:
\begin{eqnarray} \label{app-1}
u_i := U(s+\epsilon_i e_i) \approx U(s) + \frac{\partial U(s)}{\partial s_i} \epsilon_i \nonumber \\
\Rightarrow  \frac{\partial U(s)}{\partial s_i} \approx \frac{u_i-u_0}{\epsilon_i},  \ \ \ \ \  u_0:= U(s).
\end{eqnarray}
Assume that we have $m+1$ points $s$ and $s+\epsilon_i e_i$, $i \in \{1,\ldots,m\}$. By the above equations, if we have the value of $U(s)$ at these points, we can find the approximate gradient. But in the absence of true utility function, we have to find these values through proper questions from the DM. Here, we assume that we can ask the DM about the relative preference for the value of the function at these $m+1$ points. For example, DM can use a method called Analytic Hierarchy Process (AHP) to assess relative preference. We use these relative preferences to find the approximate gradient.

Assume that the DM provides us with the priority vector $p$, then we have the following relationship between $p$ and $u_i$'s
\begin{eqnarray} \label{app-2}
&&\frac{u_i}{u_j}=\frac{p_i}{p_j}, \ \ \ \ i,j \in \{0,\ldots,m\}, \nonumber \\
\Rightarrow&&\frac{u_i-u_0}{u_0}=\frac{p_i-p_0}{p_0}, \nonumber \\
\Rightarrow&& u_i-u_0 = \beta_0 (p_i-p_0),    \ \ \ \ \beta_0:=\frac{u_0}{p_0} .
\end{eqnarray}
Now, we can substitute the values of $u_i-u_0$ from \eqref{app-2} into \eqref{app-1} and we have
\begin{eqnarray} \label{app-3}
\nabla U(s)=\beta_0 \left [\frac{p_1-p_0}{\epsilon_1} \ \ \ \cdots \ \ \ \frac{p_m-p_0}{\epsilon_m} \right ]^{\top}.
\end{eqnarray}
The problem here is that we do not have the parameter $\beta_0$. However, this parameter is not important in our algorithm because we are looking for normals to our proper hyperplanes and, as it can be seen in Propositions \ref{W-l5} and \ref{W-l6}, a scaled gradient vector can also be used to calculate $h^0$ and $h^1$.  Therefore, we can simply ignore $\beta_0$ in our algorithm.

Note that supergradients may be approximate due to the imperfect nature of the interaction
with the DM.  However, this issue gives us an opportunity to highlight
another advantage of our approach when
compared to classical robust optimization.  Small errors in the
determination of uncertainty
regions in classical robust optimization may change the solution set
rather dramatically or even
make the underlying problem infeasible.  In our approach, however flawed
the supergradient information, as long as the halfspace defined by it
contains an optimal solution or an approximately optimal solution, our
algorithms are guaranteed to perform well. Therefore, this approximation in data is more serious in classical robust optimization that the DM needs to specify the whole uncertainty region, versus our approach wherein the supergradient is basically the normal to the halfspace used in reducing
the set of weights under consideration.

\subsection{Convergence of the algorithm} \label{convergence}
In this subsection, we focus on the convergence results for the $w$-space algorithm as our main algorithm. Note that no matter what the problem is, our cutting plane algorithm is applied to the unit simplex in the $w$-space. This makes the analysis straightforward and lets us use many results from the literature. On the other hand, if we use the driving factor scheme introduced above, our weight space has always dimension $k \leq 20$ and cutting plane algorithms become really fast. We define $W_O$ as the set of all weights $w$ such that the weighted center of $w$ is acceptable to the DM. In all cutting plane algorithms, a ``center" of the shrunken space must be chosen as the test point, which is crucial in the speed and convergence results of cutting plane algorithms.

Introduction of cutting-plane algorithms goes back at least to the 1960's and one of the first appealing ones is the center of gravity version \cite{gravity-1}. The center of gravity algorithm has not been used in practice because computing the center of gravity, in general, is difficult. However, it is noteworthy due to its theoretical properties.  For example,
Gr\"{u}nbaum \cite{grun} proved that by using any cutting-plane
through the center, more than $1/\exp(1) \approx 0.3678$ of the feasible set is cut out. Such results guarantee a geometric convergence rate with a sizeable constant \cite{Khachiyan,Grotschel}.
Many different types of centers have been proposed in the literature.
A group of algorithms use the center of a specific localization set,
which is updated at each step. One of them is the ellipsoid method
\cite{ellip}  where the localization set is represented by an ellipsoid
containing an optimal solution. Ellipsoid method can be related to our algorithm as we can use it to find the new weight vectors at each iteration.
Another family of cutting-plane algorithms are based on \emph{volumetric barriers} or \emph{volumetric centers} \cite{Vaidya, Vaidya-Atkin, Anst}. Vaidya used the volumetric center to design a new algorithm for minimizing a convex
function over a convex set \cite{Vaidya}. More sophisticated algorithms
have been developed based on Vaidya's volumetric cutting plane method \cite{Vaidya-Atkin, Anst}. Let us summarize  the above discussions about the three centering methods in a theorem:
\begin{theorem}  \label{thm:con-1}
	Assume that at Step 5 of the $w$-space algorithm, we set $w^{k+1}$ as one of the three following centers of $W^{k+1}$:
	\begin{itemize}
		\item the center of gravity, 
		\item the center of the minimum volume ellipsoid containing $W^{k+1}$, 
		\item the volumetric center. 
	\end{itemize}
	Also assume that $W_O$ contains a ball of radius $\epsilon$. Then, using the driving factor approach with $k=O(1)$, the algorithm stops in $O\left(\ln\left(\frac 1\epsilon \right)\right)$ iterations, with a solution acceptable to the DM. 
\end{theorem}
The cutting-plane method which is most relevant to our algorithm and is more efficient in practice is the analytic-center cutting plane method (ACCPM), see \cite{ana-sur} for a survey. In this method, the new point at each iteration is an approximate analytic center of the remaining polytope.
The complexity of such algorithms has been widely studied in the literature \cite{ana-1,ana-2}. Let us prove the following theorem: 
\begin{theorem}\label{thm:con-2}
	Assume that at Step 5 of the $w$-space algorithm, we calculate $w^{k+1}$ as the analytic center of $W^{k+1}$. Also assume that $W_O$ contains a ball of radius $\epsilon$. Then, using the driving factor approach with $k=O(1)$, the algorithm stops in $O^*\left(\frac{1}{\epsilon^2}\right)$ iterations with a solution acceptable to the DM, where $O^*$ means we ignore some logarithmic terms. 
\end{theorem}
\begin{proof}
	We use existing proved results in \cite{ana-2} and  \cite{ana-sur}.
	Consider the proof in Section 4 of \cite{ana-sur} for feasibility version of the analytic center cutting plane algorithm. The considered problem is finding $w \in C \cap [0,1]^n$, where $C$ is a closed convex set and contains a ball of radius $\epsilon$. $C$ is also equipped with an oracle that returns a cutting plane $\langle a, w-\bar w \rangle \geq 0$ whenever $\bar w \notin C$. 
	Note that we designed our approach so that our weight vectors are from the unit simplex, so $0 \leq w \leq e$. If we let $C=W_O$, it has all the required properties we mentioned. Therefore, all the discussions are carried forward and  we have the $O^*\left(\frac{n^2}{\epsilon^2}\right)$ iterations bound. However, as we use the driving factor approach, we further have $n=k=O(1)$. Hence, for our approach, the complexity bound is $O^*\left(\frac{1}{\epsilon^2}\right)$. 
\end{proof}
An alternative way to interpret the convergence properties above is after at most $O\left(\ln\left(\frac 1\epsilon \right)\right)$ iterations (ellipsoidal center algorithm) or $O^*\left(\frac{1}{\epsilon^2}\right)$ iterations (analytic center algorithm) our current iterate is within an $\epsilon$-neighborhood of a weight vector in $W_O$.  \cite{ana-sur} also has a discussion on how to modify the complexity if we use the approximate analytic center. Note that by some cut elimination and a complicated analysis, the authors in \cite{atkinson} proved a stronger convergence result for ACCPM that we mention for our approach in the following remark:
\begin{remark}
	If we use the cut elimination approach of \cite{atkinson}, we can improve the convergence result in Theorem \ref{thm:con-2} to $O\left(\ln^2\left(\frac 1\epsilon \right)\right)$ iterations. 
\end{remark}

\section{Illustrative preliminary computational experiments} \label{num}
In this section, we present some numerical results to illustrate the performance of the algorithms in the $w$-space designed in Section \ref{alg}. As we mentioned in previous sections, the utility function is not assumed to be explicitly available in our approach. So, for computational experiments with our algorithms, we maintain the same assumption. We choose a utility function; however, the algorithm does not ``see" the utility function we chose. The algorithm interacts with the utility function only through the supergradient oracle.   LP problems we use are chosen from the NETLIB library of LPs. Most of these LP problems are not in the format we have used throughout the paper which is the standard inequality form. Hence, we convert each problem to the standard equality form and then use the dual problem. In this section, the problem $\max \{ (c^{(0)})^{\top}x :  \ Ax \leq b^{(0)} \}$ is the converted one. In the following, we consider several numerical examples.

\textbf{Example 1:} In this example, we consider a simple problem of maximizing a quadratic function.  Consider the ADLITTLE problem (in the converted form) with 139 constraints and 56 variables. We apply the algorithm to function $U_{ij}(s)=-(s_i-s_j)^2$ which makes two slack variables as close as possible. This function may not have any practical application, however, shows a simple example difficult to solve by classical robust optimization.

The stopping criteria is $\|g\| \leq 10^{-6}$. For $U_{23}$ the algorithm takes 36 iterations and returns \mbox{ $U_{23}=-5 \times10^{-11}$}. For $U_{34}$ the algorithm takes 35 iterations and returns $U_{34}=-2.4\times10^{-12}$.

\textbf{Example 2:} Consider the ADLITTLE problem and assume that three constraints $\{68,71,74\}$ are important for the DM. Assume that the DM estimates that there is $20$ percent uncertainty in the RHS of these inequalities. We have $(b_{68},b_{71},b_{74})=(500,493,506)$ and so the desired slack variables are around $(s_{68},s_{71},s_{74})=(100,98,101)$.
By using the classical robust optimization method that satisfies the worst case scenario, the optimal objective value is $obj_{c}=1.6894 \times 10^5$.

Now assume that the following utility function represents DM's preferences:
$$
U_1(s)=t_{68}\ln(s_{68})+t_{71}\ln(s_{71})+t_{74}\ln(s_{74})+t_{m}\ln(s_{m}).
$$
This function is a NDAS function that we defined in Definition \ref{NDAS}. Assume that the DM set $t_m=10$ and \mbox{$t_{68}=t_{71}=t_{74}=1$}. By using our algorithm, we get the objective value of \mbox{$obj_1=1.7137 \times 10^5$} with the slack variables
$(s_{68},s_{71},s_{74})=(82,83,132)$. As we observe, the objective value is higher than the classical robust optimization method while two of the slack conditions are not satisfied. However, the slack variables are close to the desired ones. If the
DM sets $t_m=20$, we get the objective value of $obj_2=1.9694 \times 10^5$ with the slack variables
$(s_{68},s_{71},s_{74})=(40,41,79)$. However, all the iterates might be interesting for the DM. The following results are also returned by the algorithm before the optimal one:
\begin{eqnarray*}
&&obj_3=1.8847 \times 10^5, \ \  (s_{68},s_{71},s_{74})=(56,58,83), \\
&&obj_4=1.7 \times 10^5, \ \  (s_{68},s_{71},s_{74})=(82,84,125).
\end{eqnarray*}

Now assume that the DM wants to put more weight on constraints 68 and 71 and so set \mbox{$t_{68}=t_{71}=2$}, $t_{74}=1$ and $t_m=20$.
In this case, the algorithm returns $obj_5=1.8026 \times 10^5$ with the slack variables
$(s_{68},s_{71},s_{74})=(82,84,64)$.

\textbf{Example 3:} In this example, we consider the DEGEN2 problem (in the converted form) with 757 constraints and 442 variables. The optimal solution of this LP is $obj_1=-1.4352 \times 10^{3}$. Assume that constraints 245, 246, and 247 are important for the DM who wants them as large as possible, however, at the optimal solution we have $s(245)=s(246)=s(247)=0$.
The DM also wants the optimal objective value to be at least $-1.5 \times 10^{3}$. As we stated before, we add the objective function as a constraint to the system. To have the objective value at least $-1.5 \times 10^{3}$, we can add this constraint as $c^{\top}x=-1500+s_{m+1}$. For the utility function, the DM can use the NDAS function
$$
U(s)=\ln(s_{245})+\ln(s_{246})+\ln(s_{247}).
$$
By running the algorithm for the above utility function,
we get \\ $(s_{245},s_{246},s_{247})=(7.75,17.31,17.8)$ with objective value $obj_2 \approx -1500$ after 50 iterations and \\ $(s_{245},s_{246},s_{247})=(15.6,27.58,27.58)$ with $obj_3 \approx -1500$ after 100 iterations.

\textbf{Example 4:}  In this example, we consider utility functions introduced at the end of Appendix \ref{prob}. Consider problem SCORPION with optimal objective value of $obj_1=1.8781 \times 10^{3}$.  Assume that the uncertainty in constraints 211 to 215 are important for the DM and we have $\|\Delta b_i \|_1 = 0.7 b_i^{(0)}$, $i \in \{211, \ldots, 215\}$, where $\Delta b_i$ was defined in \eqref{uncer_set}. Let $\hat{x}$ be the solution of MATLAB's LP solver, then we have $s_{211}= \cdots=s_{215} =0$ which is not satisfactory for the DM.
Besides, assume that the DM wants the objective value to be at least $1800$. To satisfy that, we add the $(m+1)$th constraint as $s_{m+1}=-1800+(c^{(0)})^{\top}x$ which guarantees $(c^{(0)})^{\top}x \geq 1800$. For the utility function, first we define $u_i(s_i)$,  $i \in \{211, \ldots, 215\}$ similar to Figure \ref{Fig-utility} with $\epsilon_i^1 = \|\Delta b_i \|_1 = 0.7 b_i^{(0)}$ and $\epsilon_i^2 = \infty$. So we have for $i \in \{211, \ldots, 215\}$:
\begin{eqnarray}
u_i(s_i)=
\left\{
\begin{array}{lr}
s_i &  \ \ \ s_i < \|\Delta b_i \|_1 \\
\|\Delta b_i \|_1 & \ \ \ s_i \geq \|\Delta b_i \|_1.
\end{array}
\right.
\end{eqnarray}
Now, we can define $U(s) := \sum_{i=211}^{215} \ln u_i(s_i)$. By running the algorithm, the supergradient goes to zero after 65 iterations and the algorithm stops. Denote the solution by $x^*$, then the results are as follows:
\begin{eqnarray}
&& (c^{(0)})^{\top} x^* =  1800.3, \nonumber \\
&& b_{211}^{(0)}= 3.86, \ \ b_{212}^{(0)}= 48.26, \ \ b_{211}^{(0)}= 21.81, \ \ b_{211}^{(0)}= 48.26, \ \ b_{211}^{(0)}= 3.86,  \nonumber  \\
&& s_{211}^*= 3.29, \ \ s_{212}^*= 19.47, \ \ s_{211}^*= 7.39, \ \ s_{211}^*= 16.97, \ \ s_{211}^*= 3.24.
\end{eqnarray}
Now, assume that the DM wants the objective value to be at least $1850$ and
the $(m+1)$th constraint becomes $s_{m+1}=-1850+(c^{(0)})^{\top}x$. In
this case, the norm of the supergradient reaches zero, after 104 iterations.
The norm of supergradients versus the number of iterations are shown in
Figure \ref{Fig_numerical_2} for these two cases.
Denote the solution after 100 iterations by $\bar{x}^*$, then we have:
\begin{eqnarray}
&& (c^{(0)})^{\top} \bar{x}^* =  1850, \nonumber \\
&& \bar{s}_{211}^*= 1.22, \ \ \bar{s}_{212}^*= 16.74, \ \ \bar{s}_{211}^*= 6.80, \ \ \bar{s}_{211}^*= 14.54, \ \ \bar{s}_{211}^*= 1.25.
\end{eqnarray}
\begin{figure}[ht]
\centering
\includegraphics[height=3.1in, width = 4.5in]{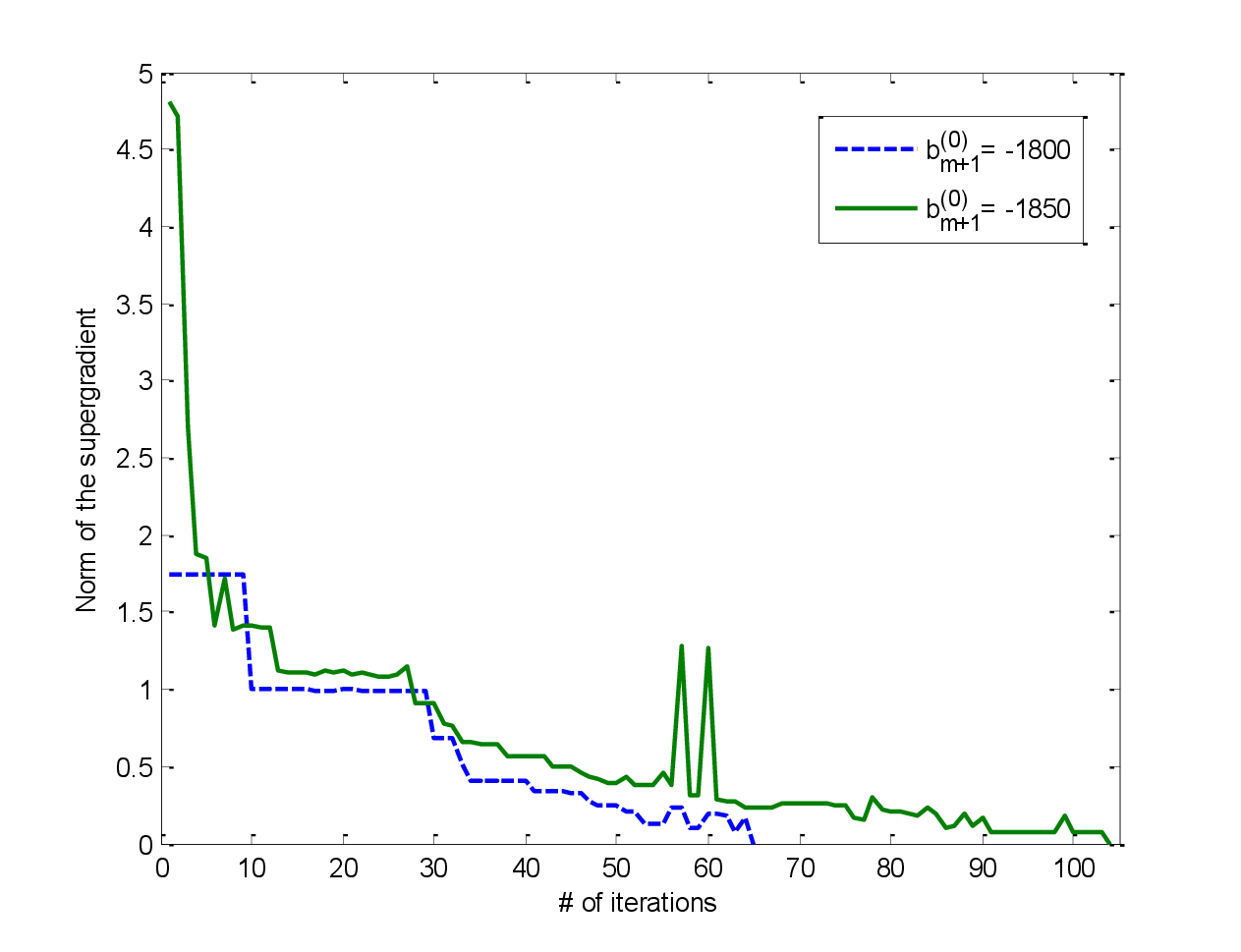}
\caption {Norm of the supergradient versus the number of iterations for Example 5.}
\label{Fig_numerical_2}
\end{figure}
Let $\bar{x}$ be the returned value in the second case
after 65 iterations. It is clearly not robust feasible;
however, we can use bound \eqref{B2} to find an upper bound
on the probability of infeasibility. Assume that $N=10$
and all the entries of $\Delta b_i$ are equal. Then,
bound \eqref{B2} reduces to $B(N,\delta_iN)$,
where $\delta_i=\frac{s_i}{ \|\Delta b_i \|_1}$.
The probabilities of infeasibility of $\bar{x}$
for constraints 211 to 215 are given in Table \ref{Table1}
(using bound \eqref{B2}).

\begin{table} [ht]
\label{Table1}
\begin{center}
\begin{tabular}{|r|r|r|}
\hline $i$& $\text{Pr}(\langle a_j,\bar{x} \rangle > \tilde{b}_j)$ \\
\hline $211$ &$0$ \\
\hline $212$ &$0.0827$ \\
\hline $213$ &$0.0018$ \\
\hline $214$ &$0.0866$ \\
\hline $215$ &$0$ \\
\hline
\end{tabular}
\caption{The probability of infeasibility of $\bar{x}$ for constraints 211 to 215.}
\end{center}
\end{table}

\section{Extensions and conclusion} \label{con}

\subsection{Extension to Semidefinite Optimization (SDP)} \label{sdp}

Semidefinite Programming is a special case of Conic Programming
where the cone is a direct product of semidefinite cones.
Many convex optimization problems can be modeled by SDP.
Since our method is based on a barrier function for a
polytope in $\mathbb R^{n}$, it can be generalized and used as an
approximation method for robust semidefinite programming that is
$NP$-hard for ellipsoidal uncertainty sets.
An SDP problem can be formulated as follows
\begin{eqnarray}
 & \sup & \langle \tilde{c},x\rangle,\nonumber \\
 &\text{s.t.}& \sum_{j=1}^{t_i}A^{(j)}_{i}x_{j}+S_{i}=\tilde{B}_{i},\ \ \ \forall i\in\{1,2,...,m\},\nonumber\\
 && S_{i}\succeq 0,\ \ \forall i\in\{1,2,\ldots,m\},\nonumber
\end{eqnarray}
where $A^{(j)}_{i}$ and $\tilde{B}_{i}$ are symmetric matrices of appropriate size,
and $\succeq$ is the L\"{o}wner order; for two square, symmetric matrices $C_1$ and $C_2$
with the same size, we have $C_1 \succeq C_2$ iff $C_1-C_2$ is a semidefinite matrix.
For every $i \in \{1, \ldots, m\}$, define
\begin{eqnarray*}
\mathcal F_i:= \{ x\in \mathbb R^n : \sum_{j=1}^{t_i}A^{(j)}_{i}x_{j} \preceq \tilde{B}_{i}\}.
\end{eqnarray*}
Assume that $\inte(\mathcal F_i) \neq \emptyset$ and let
$F_i : \inte(\mathcal F_i) \rightarrow \mathcal R$ be
a self-concordant barrier for $\mathcal F_i$. The typical
self-concordant barrier for SDP is
$F_i(x)=-\ln \left (\det \left (\tilde{B}_{i} - \sum_{j=1}^{t_i}A^{(j)}_{i}x_{j} \right) \right) $.  Assume
$$
\mathcal F := \bigcap_{i=1}^{m} \mathcal F_i
$$
is bounded and its interior is nonempty. Now,
as in the definition of the weighted center for LP, we can define a weighted center for SDP. For every $w \in \mathbb R^m_{++} $, we can define the weighted center as follows:
\begin{eqnarray}
\arg \min \left \{\sum_{i=1}^{m}w_i F_i(x) : x \in \mathcal F \right \}
\end{eqnarray}
The problem with this definition is that we do not have
many of the interesting properties we proved for LP.
The main one is that the weighted centers do not cover the
relative interior of the whole feasible region and
we cannot sweep the whole feasible region by moving
in the $w$-space. There are other notions of weighted centers
that address this problem; however, they are more difficult
to work with algorithmically. Extending the results we derived
for LP to SDP can be a good future research direction to follow.

\subsection{Quasi-concave utility functions} \label{quasi-concave}

The definition of the quasi-concave function is as follows:

\begin{definition}
A function $f : \ \mathbb R^n \rightarrow \mathbb R$ is \emph{quasi-concave} if its domain is convex, and for every $\alpha \in \mathbb R$, the set
$$
\{ x \in \dom f \ : \ f(x) \geq \alpha\}
$$
is also convex.
\end{definition}
All concave functions are quasi-concave, however, the converse is not true. Quasi-concave functions are important in many fields such as game theory and economics. In microeconomics, many utility functions are modeled as quasi-concave functions. For differentiable functions, we have the following useful proposition:
\begin{proposition}
A differentiable function $f$ is quasi-concave if and only if the domain
of $f$ is convex and for every $x$ and $y$ in $\dom f$ we have:
\begin{eqnarray} \label{quasi}
f(y) \geq f(x) \ \ \Rightarrow  \ \ (\nabla f(x))^{\top}(y-x) \geq 0
\end{eqnarray}
\end{proposition}
\eqref{quasi} is similar to \eqref{for-quasi}, which is the property of the supergradient we used to design our algorithms. The whole point is that for a differentiable quasi-concave function $U(s)$ and any arbitrary point $s^0$, the maximizers of $U(s)$ are in the half-space \\ $(\nabla U(s^0))^{\top}(s-s^0) \geq 0$. This means that we can extend our algorithms to
differentiable quasi-concave utility functions simply by replacing supergradient with gradient, and all the results for $s$-space and $w$-space stay valid.

\subsection{Conclusion}
In this paper, we presented new algorithms in a framework for robust optimization designed to mitigate some of
the major drawbacks of robust optimization in practice. Our algorithms have the potential of increasing the applicability of robust optimization. Some of the advantages of our new algorithms are:

\begin{enumerate}
\item{Instead of a single, isolated, and very demanding interaction with the DM,
our algorithms interact continuously with the DM throughout the optimization process with
 more reasonable demands from the DM in each iteration. One of the benefits of our approach
is that the DM ``learns" what is feasible to achieve throughout the process. Another benefit is that
the DM is more likely to be satisfied (or at least be content) with the final solution. Moreover, being personally
involved in the production of the final solution, the DM bears some responsibility for it and is more likely
to adapt it in practice.}

\item{Our algorithms operate in the weight-space using only driving factors with the DM.
This helps reduce the dimension of the problem,
simplify the demands on the DM while computing the most important aspect of the problem at hand.}

\item{Weight-space and weighted-analytic-centers approach embeds a ``highly differentiable" structure into the
algorithms. Such tools are extremely useful in both the theory and applications of optimization. In contrast,
classical robust optimization and other competing techniques usually end up delivering a final solution where
differentiability cannot be expected.}
\end{enumerate}

Note that many elements of our approach can be partly utilized in other
approaches to robust optimization and decision making situations under
uncertainty.  Moreover, our work creates natural connections between robust optimization and
multi-attribute utility theory, elicitation methods used
in multi-criteria decision making problems and goal programming theory
(see \cite{KeeneyRaiffa1976,MorganHenrion1990,Ignizio1976}).

Developing similar algorithms for semidefinite programming is left as a future research topic.
As we explained in Subsection \ref{sdp}, we can define a similar notion of weighted center for SDP.
However, these weighted centers do not have some of the properties we used for LP, and we may have to
switch to other notions of weighted centers that are more difficult to work with algorithmically,
and have fewer desired properties compared to the LP setting.

\bibliographystyle{plain}




\nocite{*}

\appendix


\section{Probabilistic Analysis} \label{prob}

Probabilistic analysis is tied to robust optimization. One of the recent trends in robust optimization research is the attempt to try reducing conservatism to get better results, and at the same time keeping a good level of robustness. In other words, we have to show that our proposed answer has a low probability of infeasibility. In this section, we derive some probability bounds for our algorithms based on weight and slack vectors.  These bounds can be given to the DM with each answer and the DM can use them to improve the next feedback.

\subsection{Representing the robust feasible region with weight vectors} \label{weight-1}

Before starting the probabilistic analysis, want to relate the notion of weights to the parameters of the
uncertainty set. As we explained in Subsection \ref{notions}, we consider our uncertainty sets as follows:
\begin{eqnarray} \label{uncer_set}
B_i:=\left\{\tilde{b}_i \ : \ \exists \tilde{z} = (\tilde{z}_i^1,\ldots,\tilde{z}_i^{N_i})\in [-1,1]^{N_i} \ s.t. \
\tilde{b}_i=b_i^{(0)}+\sum_{l=1}^{N_i}{\Delta b_i^l \tilde{z}_i^l} \right \},
\end{eqnarray}
where $\{ \tilde{z}_i^l \}_{l=1}^{N_i}$, $i \in \{1,\ldots,m\}$ are independent random variables, and $\Delta b_i^l$ is the scaling factor of $\tilde{z}_i^l$. We assume that the support of $\tilde{z}_i^l$ contains $\tilde{z}_i^l=-1$, i.e., $Pr\{\tilde{z}_i^l=-1\} \ne 0$. Let us define another set which is related to the weight vectors:
\begin{eqnarray} \label{W}
\mathcal W := \left \{(w_1,\ldots,w_m): w_i \in [{y}_i(w) \| \Delta b_i \| _1, 1), \  \sum_{i=1}^m {w_i} = 1 \right \},
\end{eqnarray}
where $y(w)$ is the $y$-vector of $w$. Our goal is to explicitly specify a set of weights whose corresponding $w$-center makes the
feasible solution of the robust counterpart.
\begin {proposition}\label{22}
Let $x$ satisfy $Ax \leq \tilde{b}$ for every $\tilde{b} \in B_1 \times B_2 \times \cdots \times B_m$. Then there exists
some $w \in \mathcal W$, so that $x$ is the weighted analytic center with respect to the weight vector $w$, i.e.,
$x=x(w)$. In other words,
$$
\left \{ x \ : \ Ax \leq \tilde{b},\  \forall \tilde{b} \in B_1 \times B_2 \times \cdots \times B_m \right \}
\subseteq
\left \{ x(w) \ : \ w \in \mathcal W \right \}.
$$
\end {proposition}
\begin {proof}
Let $\hat{w}>0$ be an arbitrary vector such that $\sum_{i=1}^m \hat{w}_i =1$, and let $(\hat{x},\hat{y},\hat{s})$ be the weighted center corresponding to it. Assume that $x$ is in the robust feasible region; we must have $ \langle a_i,x \rangle \leq
b_i^{(0)} + \langle \Delta b_i,\tilde{z}_i \rangle$ for every $\tilde{z}_i$ with nonzero probability, particularly for
$\tilde{z}_i= -e$ where $e$ is all ones vector. So
$$
\langle a_i,x \rangle -  b_i^{(0)} \leq \langle \Delta b_i,\tilde{z}_i \rangle
=\langle \Delta b_i, - e \rangle = -  \| \Delta b_i \|.
$$
Define $s_i :=  b_i^{(0)} -  \langle a_i,x \rangle$. Thus, from the above equation, for every $i\in \{1,\ldots,m\}$ we have
$$
0 <  \| \Delta b_i \| _1  \leq s_i,
$$
and consequently $\hat{y}_i  \| \Delta b_i \| _1  \leq \hat{y}_i s_i$ using
the fact that $\hat{y}_i > 0$. For every $i\in \{1,\ldots,m\}$, we set
$$
w_i := \hat{y}_i s_i.
$$
Since $(x,\hat{y}, s)$ satisfies the optimality conditions, we have $x=x(w)$. It remains
to show that $w \in \mathcal W $. First note that:
$$
\sum_{i=1}^{m} w_i = \sum_{i=1}^{m} s_i \hat{y}_i = \sum_{i=1}^{m} \hat{s}_i \hat{y}_i = \sum_{i=1}^{m} \hat{w}_i=1,
$$
where for the second equality we used Lemma \ref{first-l}. Now, using the fact that $w_i \geq 0$ for every $i\in \{1,\ldots,m\}$, we have $w_i < \sum_{j=1}^m w_j =1$. We already proved that
$\hat{y}_i  \| \Delta b_i \| _1  \leq \hat{y}_i s_i=w_i$. These two inequalities prove that $w_i \in [ \hat{y}_i \| \Delta b_i \| _1, 1)$.
\end {proof}
The above proposition shows that when the robust counterpart problem with respect to the uncertainty set
$B_1 \times B_2 \times \cdots \times B_m$ is feasible, the set $\mathcal W$ is nonempty. In the next proposition we prove that the equality holds in the above inclusion.
\begin{proposition}\label{33}
(a)We have
$$
\{ x \ : \ Ax \leq \tilde{b},\  \forall \tilde{b} \in B_1 \times B_2 \times \cdots \times B_m \}
= \{x(w) \ : \ w \in \mathcal W \}.
$$
\\
(b) Assume that ${w}>0$ satisfies $\sum_{i=1}^m {w}_i =1$, and $y$ is its corresponding $y$-vector. For every $i \in \{1,\ldots,m\}$, we have \\
 $$w_i \geq {y}_i \|\Delta b_i \|_1 \Rightarrow \langle a_i,x(w) \rangle \leq \tilde{b}_i, \ \ \ \forall \tilde{b}_i \in B_i.$$
\end{proposition}
\begin{proof}
(a) $\subseteq$ part was proved in Proposition \ref{22}. For $\supseteq$, let $w \in \mathcal W$ and $({x},{y},{s})$ be its corresponding weighted center. By $w \in \mathcal W$ we have
$$
 {y}_i \| \Delta b_i \|_1 \leq w_i = s_iy_i =
(b_i^{(0)} -  \langle a_i,x \rangle) {y}_i \Longrightarrow
 \| \Delta b_i \|_1 \leq (b_i^{(0)} -  \langle a_i,x \rangle).
$$
Therefore, for all $\tilde{z}_i \in \times_{i=1}^m [-1,1]$,
$$
\langle a_i,x \rangle \leq b_i^{(0)}- \| \Delta b_i \|_1
\leq b_i^{(0)}- \sum_{l=1}^N{\Delta b_i^l \ \tilde{z}_i^l} = b_i^{(0)} + \langle \tilde{z}_i,\Delta b_i \rangle,
$$
which proves $x$ is a robust feasible solution with respect to the uncertainty set $B_1 \times B_2 \times \cdots \times B_m$. \\
(b) \
Assume that ${w}>0$ satisfies $\sum_{i=1}^m {w}_i =1$, $y$ is its corresponding $y$-vector, and there exists $i \in \{1,\ldots,m\}$ such that $w_i \geq {y}_i \|\Delta b_i \|_1$. If there exists $\tilde{b}_i \in B_i$ such that $\langle a_i,x(w) \rangle > \tilde{b}_i$ where $\tilde{b}_i=b_i^{(0)}+\sum_{l=1}^{N_i}{\Delta b_i^l \tilde{z}_i^l}$, by using $\tilde{z}_i^l \geq -1$ we have
\begin{eqnarray*}
\langle a_i,x(w) \rangle > \tilde{b}_i \ &\Rightarrow& \ \langle a_i,x(w) \rangle >  b_i^{(0)}+ \sum_{l=1}^{N_i}{\Delta b_i^l \ \tilde{z}_i^l}
\ \geq \ b_i^{(0)} - \sum_{l=1}^{N_i}{\Delta b_i^l} \\
&\Rightarrow&   \sum_{l=1}^{N_i}{\Delta b_i^l} > b_i^{(0)}- \langle a_i,x(w) \rangle = s_i(w) \\
&\Rightarrow&  y_i \sum_{l=1}^{N_i}{\Delta b_i^l} > y_i s_i(w) = w_i \geq {y}_i \sum_{l=1}^{N_i}{\Delta b_i^l} \\
&\Rightarrow& \sum_{l=1}^{N_i}{\Delta b_i^l} \ > \   \sum_{l=1}^{N_i}{\Delta b_i^l},
\end{eqnarray*}
which is a contradiction. We conclude that $\langle a_i,x(w) \rangle \leq \tilde{b}_i$ for all  $\tilde{b}_i \in B_i$.
\end{proof}

\subsection{Probability bounds}
Without loss of generality, we make the following assumptions on $\tilde{b}$ and $\tilde{c}$:
\begin{itemize}
\item{For every $i\in\{1,2,\ldots,m\}$, $\tilde{b}_{i}$ can be written as
$\tilde{b}_{i}=b_i^{(0)}+\sum_{l=1}^{N_i}{\Delta b_i^l \tilde{z}_i^l}$ where
$\{ \tilde{z}_i^l \}_{l=1}^{N_i}$ are independent random variables for every $i \in \{1,\ldots,m\}$.}
\item{For each $\tilde{c}_i$, $i \in \{1, \ldots, n\}$, we have $\tilde{c}_i=c^{(0)}_i+\sum_{l=1}^{N_{ic}}{\Delta c^l_i \tilde{z}_{ic}^l}$ where
$\{ \tilde{z}_{ic}^l \}_{l=1}^{N_{ic}}$ are independent random variables.}
\end{itemize}
As can be seen above, each variable $\tilde{b}_{i}$ is the summation of a nominal value $b_i^{(0)}$
with scaled random variables $\{ \tilde{z}_i^l \}_{l=1}^{N_i}$. In practice, the number of these random
variables $N_i$ is small compared to the dimension of $A$ as we explained above: each random variable $ \tilde{z}_i^l $
represents a major source of uncertainty in the system.

 Suppose we wish to find a robust feasible solution with respect to the uncertainty set $B_1 \times B_2 \times \cdots \times B_m$, where $B_i$ was defined in \eqref{uncer_set}. By Proposition \ref{33}, it is equivalent to finding the weighted center for a $w \in \mathcal W$, where $\mathcal W$ is defined in \eqref{W}. However, finding such a weight vector is not straight forward as we do not have an explicit formula for $\mathcal W$. Assume that we pick an arbitrary weight vector $w>0$ such that $\sum_{i=1}^m {w}_i =1$, with the weighted center $({x},{y},{s})$. Let us define the vector $\delta$ for $w$ as
$$
\delta_i=\frac{w_i}{y_i \|\Delta b_i \|_1}, \ \ \ i \in \{1,2,\ldots,m\},
$$
where $\Delta b_i$ was defined in \eqref{uncer_set}. For each $i \in \{1,\ldots,m\}$, if $1 \leq \delta_i$, by Proposition \ref{33}-(b) we have $\langle a_i,x(w) \rangle \leq \tilde{b}_i$ for all  $\tilde{b}_i \in B_i$. So, the problem is with the constraints that $1 > \delta_i$.
For every such constraint, we can find a bound on the probability that $\langle a_j,x(w) \rangle > \tilde{b}_j$.
As in the proof of Proposition \ref{33}-(b), in general we can write:
\begin{eqnarray} \label{m1}
\text{Pr} \{\langle a_j,x \rangle > \tilde{b}_j \} &=& \text{Pr} \left \{ -y_i \sum_{l=1}^{N_i}{\Delta b_i^l \ \tilde{z}_i^l} > w_i = y_i \delta_i \|\Delta b_i \|_1 \right \} \nonumber \\
&=& \text{Pr} \left \{ -\sum_{l=1}^{N_i}{\Delta b_i^l \ \tilde{z}_i^l} > \delta_i \|\Delta b_i \|_1 \right \}
\leq \text{exp} \left (-\frac{\delta_i^2 (\|\Delta b_i \|_1)^2}{2 \sum_{l=1}^{N_i}{(\Delta b_i^l)^2}} \right ),
\end{eqnarray}
where the last inequality is derived by using Hoeffding's inequality:
\begin{lemma}\label{H}
(\emph{Hoeffding's inequality}\cite{ineq}) Let
$v_{1},v_{2},\ldots,v_{n}$ be independent random variables with finite
first and second moments, and for every $i\in\{1,2,\ldots,n\}$,
$\tau_{i}\leq v_{i}\leq \rho_{i}$. Then for every $\varphi>0$
$$
\Pr\left \{\sum_{i=1}^{n}v_{i}-E\left(\sum_{i=1}^{n} v_{i}\right )\geq
n\varphi\right \}\leq\exp\left[\frac{-2n^{2}\varphi^{2}}{\sum_{i=1}^{n}
(\rho_{i}-\tau_{i})^{2}}\right].
$$
\end{lemma}

Bertsimas and Sim \cite{price} derived the best possible bound, i.e., a bound that is achievable. The corresponding lemma proved in \cite{price} is as follows:
\begin{lemma} \label{price-3}
(a) If $\tilde{z}_i^l$, $l \in \{1,\ldots,{N_i}\}$, are independent and symmetrically distributed random variables in $[-1,1]$, $p$ is a positive constant, and $\gamma_{il} \leq 1$, $l \in \{1,\ldots,{N_i}\}$, then
\begin{eqnarray} \label{price-3-1}
\Pr \left \{\sum_{l=1}^{N_i}{\gamma_{il} \ \tilde{z}_i^l} \geq p \right \} \leq B({N_i},p),
\end{eqnarray}
where
$$
B({N_i},p)=\frac{1}{2^{N_i}} \left[(1-\mu){{N_i} \choose {\left\lfloor \nu \right\rfloor}} +\sum_{i=\left\lfloor \nu \right\rfloor +1}^{{N_i}}  {{N_i} \choose i} \right],
$$
where $\nu:=({N_i}+p)/2$, and $\mu:= \nu - \left\lfloor \nu \right\rfloor$. \\
(b) The bound in \eqref{price-3-1} is tight for $\tilde{z}_i^l$ having a discrete probability distribution: \\ $\Pr \{\tilde{z}_i^l=1\}=\Pr\{\tilde{z}_i^l=-1 \} =1/2$, $\gamma_{il} = 1$, $l \in \{1,\ldots,{N_i}\}$, an integral value of $p \geq 1$, and $p+{N_i}$ being even.
\end{lemma}
We can use the bound for our relation \eqref{m1} as follows. Assume that $\tilde{z}_i^l$, $l \in \{1,\ldots,{N_i}\}$, are independent and symmetrically distributed random variables in $[-1,1]$. Also denote by $\text{max}(\Delta b_i)$, the maximum entry of $\Delta b_i$. Using \eqref{m1}, We can write
\begin{eqnarray} \label{B2}
\text{Pr}\{\langle a_j,x \rangle > \tilde{b}_j\}
&=& \text{Pr} \left \{\sum_{l=1}^{N_i}{\Delta b_i^l \ \tilde{z}_i^l} > \delta_i \|\Delta b_i \|_1 \right \} \nonumber \\
&\leq& \text{Pr}\left \{ \sum_{l=1}^{N_i}{\frac{\Delta b_i^l}{\text{max}(\Delta b_i)} \ \tilde{z}_i^l} \geq \delta_i \frac{\|\Delta b_i \|_1}{\text{max}(\Delta b_i)} \right \} \nonumber \\
&\leq& B \left ({N_i},\delta_i \frac{\|\Delta b_i \|_1}{\text{max}(\Delta b_i)} \right ).
\end{eqnarray}
To compare these two bounds, assume that all the entries of $\Delta b_i$ are equal.
Bound \eqref{m1} reduces to $\exp(-\delta_i^2{N_i}/2)$, and bound \eqref{B2}
reduces to $B({N_i},\delta_i{N_i})$. We can prove that bound \eqref{B2} dominates bound
\eqref{m1}.  Moreover, bound \eqref{B2} is somehow the best possible bound
as it can be achieved by a special probability distribution as in Lemma \ref{price-3}.
The above probability bounds do not take part in our algorithm explicitly.
However, for each solution, we can present these bounds to the DM and s/he can use them to improve the feedback
to the algorithm. As an example of how these bounds may be used for the DM,
we show how to construct a concave utility function $U(s)$ based on these probability bounds.
Bounds \eqref{m1} and \eqref{B2} are functions of $\delta_i=\frac{w_i}{y_i \|\Delta b_i \|_1}=\frac{s_i}{\|\Delta b_i \|_1}$ and as a result, functions of $s$.  Now, assume that based on the probability bounds, the DM defines a function $u_i(s_i)$ for each slack variable $s_i$ as shown in Figure \ref{Fig-utility}. $u_i(s_i)$ increases as $s_i$ increases, and then at the point $\epsilon_i^1$ becomes flat. At $s_i=\epsilon_i^2$ it starts to decrease to reach zero. Parameters $\epsilon_i^1$ and $\epsilon_i^2$ are specified by the DM's desired bounds. Now, we can define the utility function as $U(s):= \prod _{j=1}^{m} u_i(s_i)$. This function is not concave, but maximization of it is equivalent to the maximization of $\ln(U(s))$ which is concave.
\begin{figure}[ht]
\centering
\includegraphics[height=2in, width = 3in]{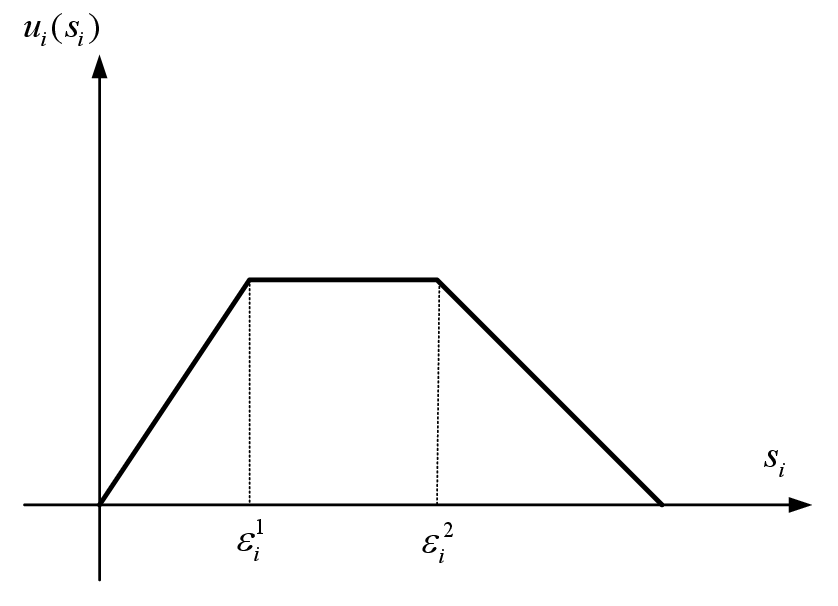}
\caption {The function $u_i(s_i)$ defined for the slack variable $s_i$}
\label{Fig-utility}
\end{figure}

\section{Properties of $w$-space} \label{appendix-1}
In this appendix, we study the properties of weight space as well as $W_s$ and $W_y$ regions. Let us start from the following well-known lemma:

\begin{lemma}\label{first-l}
Let $(x,y,s)$ and $(\hat{x},\hat{y},\hat{s})$ be the solutions of
system \eqref{analytic center} corresponding to the weight vectors
$w$, $\hat{w} \in \mathbb {R}_{++}^m$, respectively. For every $\bar{y}$
in the null space of $A^{\top}$ we have:
$$
\langle \hat{s},\bar{y}\rangle=\langle s,\bar{y}\rangle.
$$
\end{lemma}
\begin{proof}
From \eqref{analytic center}, we have $s=b-Ax$ and $\hat{s}=b-A\hat{x}$,
which results in $s-\hat{s}=A(x-\hat{x})$. Hence we have $s-\hat{s} \in \mathcal R(A)$. As the null space of $A^{\top}$ and the
range of $A$ are orthogonal, for every $\bar{y} \in \mathcal N(A^{\top})$ we can write:
$$
\langle s-\hat{s},\bar{y}\rangle = 0 \ \ \Rightarrow  \ \ \langle \hat{s},\bar{y}\rangle=\langle s,\bar{y}\rangle.
$$
\end{proof}
Let $(\hat{x},\hat{y},\hat{s})$ be the solution of system
\eqref{analytic center} corresponding to the weight vector
$\hat{w}$. Moreover, assume that $\bar{y}>0$ is such that
$A^{\top}\bar{y}=0$. Then, by using Lemma \ref{first-l}, we can show that $(\hat{x},\bar{y},\hat{s})$ is the solution
of system \eqref{analytic center} corresponding to the weight vector
$\bar Y (\hat Y)^{-1} \hat{w}$. Hence, there may be
many weight vectors that give the same $w$-center.
A stronger result is the following lemma which shows that in some cases, we can find the weighted center for a combination of weight vectors by using the combination of their weighted centers.

\begin{lemma} \label{W1}
Let $(x^{(i)},y^{(i)},s^{(i)})$, $i \in \{1,\ldots, \ell \}$, be solutions of
system \eqref{analytic center}, corresponding to the weights $w^{(i)}$. Then,
for every set of $\beta_i \in [0,1]$, $i \in \{1,\ldots, \ell\}$, such that
$\sum_{i=1}^\ell \beta_i =1$, and for every $j \in \{1,\ldots,\ell\}$, we have
$(\sum_{i=1}^\ell \beta_i x^{(i)}, y^{(j)}, \sum_{i=1}^\ell \beta_i s^{(i)})$ is the $w$-center of
$\mathcal F$, where
$$w:=\sum_{i=1}^\ell  \beta_i Y^{(j)}(Y^{(i)})^{-1} w^{(i)}.$$
Moreover,
$$
\sum_{i=1}^m w_i = \sum_{i=1}^m w_i^{(j)}.
$$
\end{lemma}
\begin{proof}
According to the assumptions, for every $i \in \{1,\ldots,\ell\}$, we have
\begin{eqnarray}
\nonumber && Ax^{(i)}+s^{(i)}=b^{(0)}, \ \ s>0,\\
\nonumber && A^{\top}y^{(i)}=0,\\
\nonumber && S^{(i)}y^{(i)}=w^{(i)}.
\end{eqnarray}
 Now, it can be seen that $(\sum_{i=1}^\ell \beta_i x^{(i)}, y^{(j)}, \sum_{i=1}^\ell \beta_i s^{(i)})$
satisfies the system:
\begin{eqnarray} \label{L1}
\nonumber && A(\sum_{i=1}^\ell \beta_i x^{(i)})+(\sum_{i=1}^\ell \beta_i s^{(i)})=b^{(0)}, \ \ (\sum_{i=1}^\ell \beta_i s^{(i)})>0,\\
\nonumber && A^{\top}y^{(j)}=0,\\
 && (\sum_{i=1}^\ell \beta_i S^{(i)}) y^{(j)} = \sum_{i=1}^\ell  \beta_i Y^{(j)}(Y^{(i)})^{-1} w^{(i)}.
\end{eqnarray}
Since the $w$-center of $\mathcal F$ is unique, the proof for the first part is done. \\
For the second part, from (\ref{L1}) we can write
\begin{eqnarray*}
\sum_{i=1}^m w_i = \sum_{i=1}^m (\sum_{p=1}^\ell \beta_p s_i^{(p)}) y_i^{(j)}=  \sum_{p=1}^\ell \beta_p ( \sum_{i=1}^m s_i^{(p)}y_i^{(j)}) =
\sum_{p=1}^\ell \beta_p \langle s^{(p)},  y^{(j)} \rangle.
\end{eqnarray*}
By Lemma \ref{first-l}, we have $\langle s^{(p)} , y^{(j)} \rangle = \langle s^{(i)}  ,y^{(j)} \rangle$. Therefore, we can continue the above series of equations as follows:
\begin{eqnarray*}
\sum_{i=1}^m w_i= \sum_{p=1}^\ell \beta_p \langle s^{(j)} , y^{(j)} \rangle=  \sum_{p=1}^\ell \beta_p ( \sum_{i=1}^m s_i^{(j)} y_i^{(j)}) = ( \sum_{i=1}^m w_i^{(j)}) \sum_{p=1}^\ell \beta_p= \sum_{i=1}^m w_i^{(j)}.
\end{eqnarray*}
\end{proof}

\subsection{Properties of $w$-space} \label{weight-2}

In this subsection, we study the structure of the $w$-space, which is important for the design of the algorithms in Section \ref{alg}. Let $s$ and $y$ be centric. First, we note that the simplex of the weight vectors can be divided into regions of constant $y$-vector  ($W_y$) and constant $s$-vector ($W_s$).
By using Lemma \ref{W1}, if $(\hat{x},\hat{y},\hat{s})$ is the solution of system \eqref{analytic center} corresponding to the weight vector $\hat{w} \in W$, and $\bar{y}>0$ is any centric $y$-vector, then $(\hat{x},\bar{y},\hat{s})$ is the solution of system \eqref{analytic center} corresponding to the weight vector $\bar Y (\hat Y)^{-1} \hat{w}$. This means that for every centric vector $\hat{s}$ and any centric vector $y$, $\hat{S}y$ is a weight vector in the simplex.

For every pair of centric vectors $s$ and $y$, $W_s$ and $W_y$ are convex. To see this, let  $(x,\bar{y},s)$ and $(x,y,s)$ be the weighted centers of $\hat{w}$ and $w$. Then, it is easy to see that for every $\beta \in [0,1]$, $(x,\beta\bar{y}+(1-\beta) y,s)$ is the weighted center of $\beta \hat{w}+(1-\beta) w$. With a similar reasoning, $W_y$ is convex for every centric $y$.

Using \eqref{analytic center}, we can express $W_s$ and $W_y$ as follows:
\begin{eqnarray} \label{W_y}
W_y = Y[(\mathcal R(A)+b) \cap \mathbb R_{++}^m] \cap B_1(0,1),
\end{eqnarray}
\begin{eqnarray} \label{W_s}
W_s = S[\mathcal N(A^{\top}) \cap \mathbb R_{++}^m] \cap B_1(0,1),
\end{eqnarray}
where $B_1(0,1)$ is the unit ball in $1$-norm centered at zero vector. Here,
we want to find another formulation for $W_y$ that might work better in some cases.
We use the following lemma.
\begin{lemma}\label{W-l4}
Assume that the rows of $B_y \in \mathbb R^{(m-n) \times m}$ make a basis for the null space of $A^{\top}Y$. Then there exists $x \in \mathbb R^{n}$ such that $YAx+w=Yb$ if and only if $B_yw=B_yYb$. I.e., $(Yb-w) \in \mathcal R(YA)$ iff $(Yb-w) \in \mathcal N(B_y)$.
\end{lemma}
\begin{proof}
Assume that there exists $x$ such that $YAx+w=Yb$. By multiplying both sides with $B_y$ from the left and using the fact that
$B_yYA=0$ we have the result. For the other direction,
assume that $B_yw=B_yYb$. Then $B_y(w-Yb)=0$ which means $w-Yb$
is in the null space of $B_y$. Then, using the orthogonal decomposition theorem,
we have $\mathcal N(B_y)=\mathcal R(B_y^{\top})^{\bot}=\mathcal N(A^{\top}Y)^{\bot}=\mathcal R(YA)$.
Thus, there exists $x$ such that $YAx+w=Yb$.
\end{proof}
Assume that $B \in \mathbb R^{(m-n) \times m}$ is such that
its rows make a basis for the null space of $A^{\top}$.
For every vector $y$, we have $A^{\top}y=A^{\top}Y(Y^{-1}y)$,
so if $y$ is in the null space of $A^{\top}$, $Y^{-1}y$ is in
the null space of $A^{\top}Y$. Hence, if the rows of $B$ make
a basis for the null space of $A^{\top}$, the rows of $BY^{-1}$
make a basis for the null space of $A^{\top}Y$ and we can write
$B_y=BY^{-1}$.  Using Lemma \ref{W-l4}, there exists $x$ such that $YAx+w=Yb$ if and only if $BY^{-1}w=BY^{-1}Yb=Bb$, and we can write \eqref{W_y} as:
\begin{eqnarray} \label{W_y2}
W_y = \left \{  w>0 \ : \ BY^{-1}w=Bb, \ e^{\top}w=1 \right \}.
\end{eqnarray}
Let us denote the affine hull with $\aff(.)$. We can prove the following lemma about $W_s$ and $W_y$.
\begin{lemma} \label{W-l1}
Assume that $s$ and $y$ are centric, we have
\begin{eqnarray*}
W_s=\textup{aff}(W_s) \cap W \ \ \textup{and}  \ \
W_y=\textup{aff}(W_y) \cap W.
\end{eqnarray*}
\end{lemma}
\begin{proof}
We prove the first one and our proof for the second one is the same. Clearly we have $W_s \subseteq \textup{aff}(W_s) \cap W$. To prove the other side, assume by contradiction that there exist $w \in \textup{aff}(W_s) \cap W$ such that $w \notin W_s$. Pick an arbitrary $\hat{w} \in \relint(W_s)$ and consider all the points $w({\beta})=\beta {w} + (1-\beta) \hat {w}$ for $\beta \in [0,1]$. Both $w$ and $\hat{w}$ are in $\textup{aff}(W_s)$, so all the points $w({\beta})$ are also in $\textup{aff}(W_s)$. $w(0) \in W_s$ and $w(1) \notin W_s$, so let $\hat{\beta}$ be $\sup \{\beta : w(\beta) \in W_s \}$.

Note that all the points in $W_s$ has the same $s$-vector, so we have $w(\beta)=Sy(\beta)$ for $\beta \in [0,\hat{\beta})$. By using \eqref{analytic center} we must also have $w(\hat{\beta}) \in W_s$. We want to prove that $\hat{\beta}=1$. Assume that $\hat{\beta}<1$. All the points on the line segment between $w(0)$ and $w(\hat{\beta})$ have the same $s$-vector and we can write them as $S(\gamma y(0) + (1-\gamma) y(\hat{\beta}))$ for $\gamma \in [0,1]$. But note that $y(\hat{\beta})>0$, so there is a small enough $\epsilon > 0$ such that $y_{\epsilon}=(-\epsilon y(0) + (1+\epsilon) y(\hat{\beta}))>0$ and hence $Sy_{\epsilon}$ is a weight vector in $W_s$. However, it is also a vector on the line segment between $w(\hat{\beta})$ and $w$ which is a contradiction to $\hat{\beta}=\sup \{\beta : w(\beta) \in W_s \}$.  So $\hat{\beta}=1$ and $w=w(1) \in W_s$ which is a contradiction. Hence $W_s \supseteq \textup{aff}(W_s) \cap W$ and we are done.
\end{proof}

We conclude that $W$ is sliced in two ways by $W_y$'s and $W_s$'s for centric $s$ and $y$ vectors. For each centric $s$ and each centric $y$, $W_y$ and $W_s$ intersect at a single point $Sy$ on the simplex. We want to prove that the smallest affine subspace containing $W_s$ and $W_y$ is \mbox{$\aff(W)=\{w : e^{\top}w=1\}$}. To that end, we prove some results on the intersection of affine subspaces. We start with the following definition:
\begin{definition} \label{lineality-d}
The \emph{recession cone} of a convex set $C \in \mathbb R^n$ is denoted by $\textup{rec}(C)$ and defined as:
$$
\textup{rec}(C):= \{y \in \mathbb R^n \ : \ (x+y) \in C, \ \ \forall x \in C\}.
$$
The \emph{lineality space} of a convex set $C$ is denoted by $\lin(C)$ and defined as:
$$
\lin(C):= (\textup{rec}(C)) \cap (-\textup{rec}(C)).
$$
\end{definition}

Let $U$ be an affine subspace of $\mathbb R^m$. If $y \in \textup{rec}(U)$,
then $-y \in \textup{rec}(U)$, which means $(\textup{rec}(U))=(-\textup{rec}(U))$.
Therefore, by Definition \ref{lineality-d}, we have
$\lin(U)=\textup{rec}(U)$. Then, by using the definition of the affine space we have:
\begin{eqnarray} \label{lineality}
\lin(U) := \{u_1-u_2 : \forall u_1,u_2 \in U \}.
\end{eqnarray}
In other words, $\textup{lin}(U)$ is a linear subspace such that $U=u+\textup{lin}(U)$ for all $u \in U$
where $'+'$ is the Minkowski sum. The following two lemmas are standard, see, for instance, \cite{Gallier}.
\begin{lemma} \label{aff-1}
Given a pair of nonempty affine subspaces $U$ and $V$ in $\mathbb R^n$, the following facts hold: \newline
(1) $U \cap V \neq \emptyset$  iff for every $u \in U$ and $v \in V$, we have $(v-u) \in \textup{lin}(U)+\textup{lin}(V)$. \newline
(2) $U \cap V$ consists of a single point iff for every $u \in U$ and $v \in V$, we have $$(v-u) \in \textup{lin}(U)+\textup{lin}(V) \ \ \textup{and} \ \ \textup{lin}(U) \cap \textup{lin}(V)=\{0\}.$$
(3) For every $u \in U$ and $v \in V$, we have
$$
\textup{lin}(\aff(U \cup V))=\textup{lin}(U)+\textup{lin}(V)+ \{ \alpha (v-u): \ \alpha \in \mathbb R \}.
$$
\end{lemma}

\begin{lemma} \label{aff-2}
Let $U$ and $V$ be nonempty affine subspaces in $\mathbb R^n$. Then we
have the following properties:

(1) if $U \cap V = \emptyset$, then
$$
\dim(\aff(U \cup V))=\dim(U)+\dim(V)+1-\dim(\textup{lin}(U) \cap \textup{lin}(V)),
$$

(2) if $U \cap V \neq \emptyset$, then
$$
\dim(\aff(U \cup V))=\dim(U)+\dim(V)-\dim(U \cap V).
$$
\end{lemma}
Using the above lemmas, we deduce the following proposition.
\begin{proposition} \label{aff-3}
Assume that $s$ and $y$ are centric $s$-vector and $y$-vector, respectively. Then the smallest affine subspace containing $W_s$ and $W_y$ is $\aff(W)=\{w : e^{\top}w=1\}$.
\end{proposition}
\begin{proof}
We assumed that $A \in \mathbb R^{m \times n}$  has full column rank, i.e., $\textup{rank}(A)=n\leq m$ and the interior of $\{x: Ax \leq b\}$ is not empty. Let $B_s$ denote the set of all centric $s$-vectors, i.e., the set of $s$-vectors for which there exist $(x,y,s)$ satisfies all the equations in \eqref{analytic center}. We claim that $B_s=\{s>0:s=b-Ax\}$. For every $s\in \{s>0:s=b-Ax\}$, pick an arbitrary $y>0$ such that $A^{\top}y=0$. For every scalar $\alpha$ we have $A^{\top}(\alpha y)=0$, so we can choose $\alpha$ such that $\alpha y^{\top}s=1$. Hence $(x,\alpha y, s)$ satisfies \eqref{analytic center} and we conclude that $B_s=\{s>0:s=b-Ax\}$. The range of $A$ has dimension $n$ and since $B_s$ is not empty; it is easy to see that the dimension of $B_s$ is also $n$. Moreover, we have $W_y=YB_s$ and since $Y$ is non-singular, we have $\dim(W_y)=n$.

Now denote by $B_y$ the set of centric $y$-vectors. By \eqref{analytic center}, we have $A^{\top}y=0$. The dimension of the null space of $A^{\top}$ is $(n-m)$. In addition, we have to consider the restriction $e^{\top}w=1$; we have
$$
1=e^{\top}w=e^{\top}(Ys)=s^{\top}y=(b-Ax)^{\top}y=b^{\top}y-x^{\top}A^{\top}y=b^{\top}y.
$$
So, we have $b^{\top}y=1$ for centric $y$-vectors which reduces the dimension by one (since $b \notin \mathcal R(A)$), and $\dim(B_y)=m-n-1$. We have $W_s=SB_y$ and so by the same explanation $\dim(W_s)=m-n-1$.

We proved that $W_s$ and $W_y$ intersect at only a single point $w=Sy$, so $\dim(W_s \cap W_y)=0$. By using Lemma \ref{aff-2}-(2) the dimension of the smallest affine subspace containing $W_s$ and $W_y$ is
$$
\dim(W_s)+\dim(W_y)-\dim(W_s \cap W_y)=n+m-n-1=m-1.
$$
The dimension of  $\aff(W)$ is also $m-1$, so by Lemma \ref{W-l1} $\aff(W)$ is the least affine subspace containing $W_s$ and $W_y$.
\end{proof}

\section{} \label{appendix}

\begin{example} \label{CoE}
The statement of Proposition \ref{W-l3} is not true for a general concave function.
\end{example}
\begin{proof}
Consider the first example of Example \ref{Wsy}. We have $m=3$, $n=1$, \\ $A=[1,\ -1,\ -1]^{\top}$, and $b=[1, \ 0, \ 0]^{\top}$. Using \eqref{analytic center}, the set of centric $s$-vectors is
$$B_s=\{[1-z, \ \ z, \ \ z]^{\top}: z\in(0,1)\}.$$
The set of centric $y$-vectors, $B_y$, is specified by solving $A^{\top}y=0$ and $y^{\top}b=1$ while $y>0$ and we can see that $B_y=\{[1, \ \ z, \ \ 1-z]^{\top}: z\in(0,1)\}$. As shown in Figure \ref{Fig2}, $W_s$ regions are parallel line segments while $W_y$ regions are line segments that all intersect at $[1, \ 0, \ 0]^{\top}$.

Now, assume that the function $U(s)$ is as follows (does not depend on $s_3$)
\begin{eqnarray} \label{CoE-1}
U(s)=\left \{
\begin{array}{llr}
3s_1-s_2, \ \ & \textup{if} \ \  & s_1 \leq s_2; \\
-s_1+3s_2, \ \ & \textup{if} \ \ & s_1 > s_2. \\
\end{array}
\right.
\end{eqnarray}
This function is piecewise linear and it is easy to see that it is concave. $U(s)$ is also differentiable at all the points except the points $s_1=s_2$. At any point that the function is differentiable, the supergradient is equal to the gradient of the function at that point. Hence, we have $\partial U(s)= \{[3, \ -1, \  0]^{\top}\}$ for $s_1 < s_2$ and $\partial U(s)= \{[-1, \ 3, \ \ 0]^{\top}\}$ for $s_1 > s_2$.

If we consider $U(s)$ on $B_s$, we can see that the maximum of the function is attained at the point that $s_1=s_2$, so \mbox{$s_{opt}=[1/2, \ 1/2,\ 1/2]^{\top}$}. Now assume that we start at \mbox{$w^0=S^0y^0=[0.4, \ 0.1, \ 0.5]^{\top}$}. Because we have $y_1=1$ for all centric $y$-vectors, $w^0_1=s^0_1$, and we can easily find $s^0$ and $y^0$ as \mbox{$s^0=[0.4, \ 0.6, \ 0.6 ]^{\top}$} and $y^0=[1, \ 1/6, \ 5/6]^{\top}$. The hyperplane passing through $w^0$ is \\ \mbox{$(g^0)^{\top}(Y^0)^{-1}(w-w^0)=0$} and since $s^0_1 < s^0_2$ we have
\begin{eqnarray} \label{CoE-2}
(g^0)^{\top}(Y^0)^{-1}=[3, \ -1, \ \ 0](Y^0)^{-1}=[3, \ -6, \ \  0],
\end{eqnarray}
and we can write the hyperplane as $3(w_1-0.4)-6(w_2-0.1)=0$. In the next step, we have to choose a point $w^1$ such that
$(g^0)^{\top}(Y^0)^{-1}(w^1-w^0) \geq 0$. Let us pick $w^1=[0.6, \ 0.19, \ 0.21]^{\top}$ for which we can easily find $s^1=[0.6, \ 0.4, \ 0.4]^{\top}$ and $y^1=[1, \ 0.475, \ 0.525]^{\top}$. For this point we have $s^1_1 > s^2_2$, so $(g^1)^{\top}(Y^1)^{-1}=[-1, \ 6.32, \  0]^{\top}$ and the hyperplane passing through $w^1$ is \mbox{$-(w_1-0.6)+6.32(w_2-0.19)=0$}. The intersection of two hyperplanes on the simplex can be found by solving the following system of equations:
\begin{eqnarray} \label{CoE-3}
\left \{
\begin{array}{l}
3w_1-6w_2=0.6 \\
-w_1-6w_2=0.6 \\
w_1+w_2+w_3=1 \\
\end{array}
\right. \Rightarrow  \ w^*=\left [
\begin{array}{l}
0.57 \\
0.185 \\
0.245 \\
\end{array}
\right].
\end{eqnarray}
The intersection of simplex and the hyperplanes $(g^0)^{\top}(Y^0)^{-1}(w-w^0)=0$ and $(g^1)^{\top}(Y^1)^{-1}(w-w^1)=0$ are shown in Figure \ref{Fig5}. The intersection of simplex with \\ \mbox{ $\{w: \ (g^0)^{\top}(Y^0)^{-1}(w-w^0)\geq 0, \ (g^1)^{\top}(Y^1)^{-1}(w-w^1)\geq 0 \}$} is shown by hatching lines. As can be seen, we have:
$$
\left \{w: \ (g^{0w})^{\top}(w-w^0) \geq 0, \ (g^{1w})^{\top}(w-w^1) \geq 0  \right \} \cap W_{s_{op}}  = \ \phi.
$$

\begin{figure}[ht]
\centering
\includegraphics[height=3in, width = 3.1in]{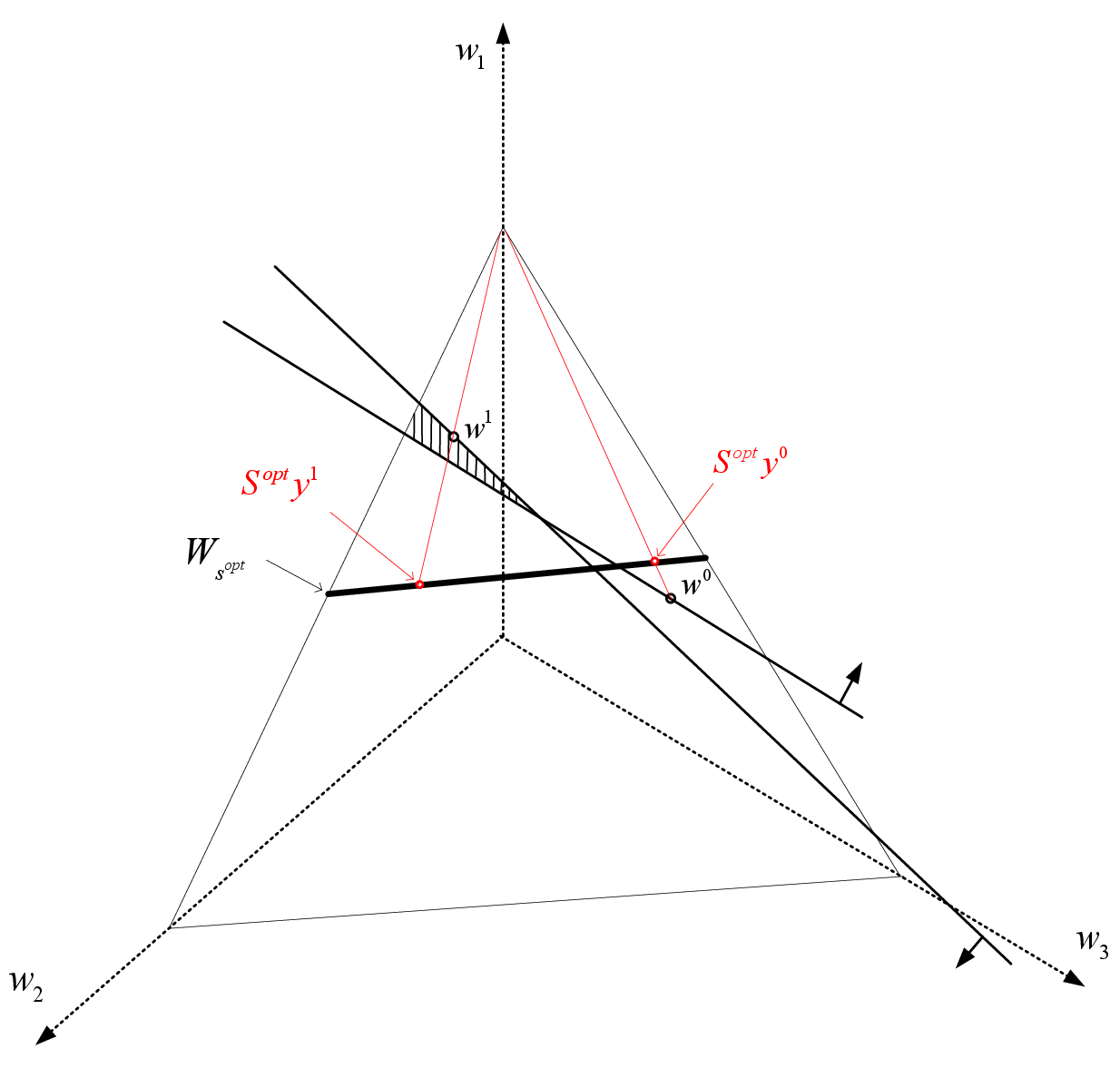}
\caption {Intersection of simplex and the hyperplanes $(g^0)^{\top}(Y^0)^{-1}(w-w^0)=0$ and $(g^1)^{\top}(Y^1)^{-1}(w-w^1)=0$ in Example \ref{CoE}.}
\label{Fig5}
\end{figure}
\end{proof}

\end{document}